\documentclass[sn-mathphys-num]{sn-jnl}
\usepackage{graphicx}%
\usepackage{multirow}%
\usepackage{amsmath,amssymb,amsfonts}%
\usepackage{amsthm}%
\usepackage{mathrsfs}%
\usepackage[title]{appendix}%
\usepackage{xcolor}%
\usepackage{textcomp}%
\usepackage{manyfoot}%
\usepackage{booktabs}%
\usepackage{algorithm}%
\usepackage{algorithmicx}%
\usepackage{algpseudocode}%
\usepackage{listings}%

\usepackage{amscd,amsxtra,array,arydshln,bbm,bm,cases,caption,color,enumerate,extarrows,float,hyperref,subfigure,verbatim}
\usepackage{amsmath,amscd,amsthm,amsxtra,amsfonts,amssymb,algorithm,algorithmicx,algpseudocode,array,arydshln,arydshln,bbm,bm,booktabs,cases,caption,color,enumerate,extarrows,float,graphicx,geometry,hyperref,listings,mathrsfs,manyfoot,multirow,subfigure,textcomp,verbatim,xcolor}

\allowdisplaybreaks

\theoremstyle{thmstyleone}%
\newtheorem{theorem}{Theorem}[section]% meant for sectionwise numbers
\newtheorem{lemma}[theorem]{Lemma}
\newtheorem*{lemma*}{Lemma}
\newtheorem{corollary}[theorem]{Corollary}

\newtheorem{definition}[theorem]{Definition}

\theoremstyle{thmstyletwo}%
\newtheorem{remark}{Remark}%

\begin{document}
\title[Non-Markovian dynamics: the memory-dependent probability density evolution equations]{Non-Markovian dynamics: the memory-dependent probability density evolution equations}
	
	
	\author[1,2,4]{\fnm{Bin} \sur{Pei}}\email{binpei@nwpu.edu.cn}
	
	\author[1]{\fnm{Lifang} \sur{Feng}}\email{FLF.fenglifang@outlook.com}
	
	\author[3]{\fnm{Yunzhang} \sur{Li}}\email{li\_yunzhang@fudan.edu.cn}
	
	\author*[1,4]{\fnm{Yong} \sur{Xu}}\email{hsux3@nwpu.edu.cn}
	
	\affil*[1]{\orgdiv{School of Mathematics and Statistics}, \orgname{Northwestern Polytechnical University}, \orgaddress{\city{Xi'an}, \postcode{710072}, \country{China}}}
	
	\affil[2]{\orgname{Research \& Development Institute of Northwestern Polytechnical University in Shenzhen}, \orgaddress{\city{Shenzhen}, \postcode{518057}, \country{China}}}
	
	\affil[3]{\orgdiv{Research Institute of Intelligent Complex Systems}, \orgname{Fudan University}, \orgaddress{\city{Shanghai}, \postcode{200433}, \country{China}}}
	
	\affil[4]{\orgdiv{MOE Key Laboratory for Complexity Science in Aerospace}, \orgname{Northwestern Polytechnical University}, \orgaddress{\city{Xi'an}, \postcode{710072}, \country{China}}}
	
	
	\abstract{This paper aims to investigate the non-Markovian dynamics. The governing equations are derived for the probability density functions (PDFs) of non-Markovian stochastic responses to Langevin equation excited by combined fractional Gaussian noise (FGN) and Gaussian white noise (GWN). The main difficulty here is that the Langevin equation excited by FGN cannot be augmented by a filter excited by GWN, leading to the inapplicability of It\^o stochastic calculus theory. Thus, in the present work, based on the fractional Wick It\^o Skorohod integral and rough path theory, a new non-Markovian probability density evolution method is established to derive theoretically the memory-dependent probability density evolution equation (PDEEs) for the PDFs of non-Markovian stochastic responses to Langevin equation excited by combined FGN and GWN, which is a breakthrough to stochastic dynamics. Then, we extend an efficient algorithm, the local discontinuous Galerkin method, to numerically solve the memory-dependent PDEEs. Remarkably, this proposed method attains a higher accuracy compared to the prevalent methods such as finite difference, path integral (PI) and Monte Carlo methods, and boasts a broader applicability than the PI method, which fails to solve the memory-dependent PDEEs.  Finally, several numerical examples are illustrated to verify the proposed scheme.}
	
	\keywords{Non-Markovian dynamics, fractional Gaussian noise, local discontinuous Galerkin, rough path theory, fractional Wick It\^o Skorohod integral}
	
	\pacs[MSC Classification]{60G22, 60L20, 82C31}
	
	\maketitle
	
\section{Introduction}\label{sec1}
The main aim of the Langevin equation is to explain the Brownian motion (BM) and its random behavior \cite{langevin1908theorie}. Without taking into account the thermal fluctuation of the particles, a particle in a fluid is subjected only to a friction force that is proportional to the velocity of the particle. The solution to the equation of motion is an exponential drop in velocity toward zero. But in nature, because of the collisions and interactions with other particles, that would not be the case. Investigating each single interaction and writing a proper equation of motion for each interaction is an impossible act due to the huge number of variables involved in it. This means that one has to introduce a force called the Langevin force to take care of all of these interactions and their random nature. The random force in the Langevin equation is usually regarded as the Gaussian white noise (GWN), i.e., the force at each point in time depends only on the immediate point before that time. Note that the GWN is not a differentiable function in the ordinary sense, which makes it difficult to express and handle the Langevin equation excited by GWN strictly mathematically. The introduction of the It\^o integral \cite{ito1951stochastic} enables the Langevin equation to be studied within the theoretical framework of It\^o stochastic calculus \cite{karatzas1991brownian,cohen2015stochastic}. This framework provides a complete set of mathematical tools and methods to deal with the Langevin equation excited by GWN, including It\^o's lemma, Feynman-Kac formula, and so on. Through these tools and methods, one can conduct more in-depth and systematic research on the Langevin equation, such as analyzing important properties like the stability, asymptotic behavior and ergodicity \cite{ottobre2011asymptotic,mao2007stochastic,kloeden1992stochastic}. The It\^o stochastic calculus theory provides a solid mathematical basis for understanding and predicting stochastic phenomena in various mechanical, physical, chemical, biological, and other systems \cite{lin1995probabilistic,coffey2004langevin}. 

However, not all random forces experienced by particle motion satisfy the assumptions that they are uncorrelated at different times \cite{hurst1951long,haunggi1994colored,chen2015memoryless,rangarajan2003processes,Li2024Transition}. Fractional Gaussian noise (FGN) is a kind of model to describe the memory. The Langevin equation excited by FGN has extensive applications in mechanics, physics, chemistry, economy and finance, and other fields \cite{lu2024response,bakalis2025barrier,chong2024frictions}. It has also been used to model the memory-enhanced stochastic sediment transport in persistent turbulent flow \cite{hung2024modeling} and ultrafast dynamics in porous media \cite{xu2020ultrafast}.
Note that the FGN has characteristics such as the long-range dependence and non-Markovian property, which make it significantly different from GWN. The It\^o stochastic calculus theoretical framework can not be applied in solving the Langevin equation excited by FGN. Thus, this paper will try to establish some new mathematical tools and methods to deal with the Langevin equation excited by FGN by using rough path theory to replace the theoretical framework of It\^o stochastic calculus.
	
To better quantify the uncertainty evolution in stochastic dynamics, the most effective way is to calculate the probability density functions (PDFs) of the Langevin equation excited by FGN, which contains the complete statistical information on the uncertainty. In this regard, there are well established formulas to write the corresponding Fokker-Planck-Kolmogorov equation (FPK) \cite{risken1996fokker,soize1994fokker} of the Langevin equation with GWN. The derivations resort to the Chapman-Kolmogorov equation as a starting point with the Markovian assumptions for the system responses. However, the FPK type equation for PDFs of the Langevin equation excited by FGN is no longer applicable due to the non-Markovian property caused by FGN. 
An alternative way to obtain the FPK type equation is by deriving the It\^o formulas, but due to the non-semi-martingale property of FGN, there is no It\^o formula of analytical form for the functions of solutions to the Langevin equation excited by FGN. The unavailability of FPK type equations poses as a significant obstacle on quantifying the uncertainty in the Langevin equation excited by FGN. In this paper, we will overcome this problem by combining the fractional Wick It\^o Skorohod (FWIS) integral \cite{ducan2000stochastic} and rough path theory \cite{lyons1998differential}. It allows us to solve the Langevin equation excited by FGN pathwise, not relying on It\^o calculus theory. That means, in the framework of the rough path theory, the Langevin equation excited by FGN can be solved in the pathwise approach by taking a realization of the driving path. The proposed approach has the advantage of offering a clear break between the deterministic rough path calculus and the Langevin equation excited by FGN.
This is a first attempt to investigate the non-Markovian dynamics by combining the FWIS integral and rough path theory.
	
Up to now, the general result is unknown, but it is necessary to mention that efforts have been made to obtain approximate FPK type equations for some special cases. For instance, Vyoral \cite{vyoral2005kolmogorov} derived the Kolmogorov backward equations for a linear stochastic differential equations (SDEs) excited by FGN, then one can obtain the FPK equations by using the relationship between the adjoint of the operators that FPK and Kolmogorov backward equations required. Other special cases are given by \cite{choi2021entropy}  and \cite{vaskovskii2022analog} where the authors have derived the governing equations for PDFs of SDEs excited by FGN, under the very stringent condition that the drift terms are missing. The papers \cite{unal2007fokker} and \cite{zeng2012fbm} try to derive the FPK equations in more geneal cases, however, \cite{zeng2012fbm} has incorrectly used the fractional It\^o formula to derive the reducibility conditions of nonlinear SDEs to linear SDEs.  The It\^o formula obtained in \cite{bender2003an} for FGN did not work for the SDEs in \cite{unal2007fokker}. Therefore, their derivation might require a modification. Thus, the acquisition for the governing equations of SDEs driven by both FGN and GWN with drift term remains open. In this paper, we mainly use the FWIS integral theory and rough path theory to develop a non-Markovian probability density evolution method (PDEM). Then, the memory-dependent probability density evolution equations (PDEEs) for the PDFs of SDEs excited by combined FGN and GWN are derived theoretically. Here, the fractional-power time varying term, $Ht^{2H-1}$ in Theorems \ref{fplinax} and \ref{fpcomm}, which reflects the presence of correlation, is an essential consequence of the correlations of FGNs. 
	
	Numerical approximation plays a crucial role in solving FPK equations since analytical solutions are usually not available. The Monte Carlo (MC) method can provide satisfactory accuracy PDFs' estimations for stochastic responses of SDEs if sufficient large sample paths are available, which means a longer calculation time is needed. In addition to the MC method, there are also some other numerical algorithms, for example, that encompass the methods of finite difference (FD)\cite{pichler2013numerical}, finite element (FE)\cite{kumar2006solution}, path integral (PI) \cite{wehner1983numerical,xu2019path}, deep learning \cite{zhang2022solving,zhang2023deep,xiao2024a}, among others.
	Traditional numerical methods struggle to achieve satisfactory high-precision solutions, while the PI method is not well-suited for non-Markovian systems.
	In this paper, the local discontinuous Galerkin (LDG) method is used to obtain the numerical solutions of exceptional precision for the memory-dependent PDEEs. The LDG method was introduced by Cockburn and Shu \cite{cockburn1998local} to solve convection-diffusion equations, inspired by the successful numerical experiments of Bassi and Rebay \cite{bassi1997high} on the compressible Navier-Stokes equations. 
	This scheme extends the discontinuous Galerkin (DG) method, originally developed by Cockburn et al. \cite{cockburn1989tvb1, cockburn1989tvb, cockburn1990runge, cockburn1998runge} for nonlinear hyperbolic systems, and retains its key advantages: facilitates the design of high-order approximations and minimizes data communication overhead, thereby enabling efficient parallel implementations and reducing computation time.
	We anticipate that the LDG method can extend beyond PI method, which is limited to the SDEs with Markovian properties, to non-Markovian SDEs while retaining the advantages and flexibility of the PI method. 
    
The aim of this paper has twofold:
		(i) the memory-dependent PDEEs are proposed for PDFs of the Langevin equations excited by combined FGN and GWN;
		(ii) an efficient LDG algorithm is developed to numerically solve the given memory-dependent PDEEs. We will develop new approaches, based on FWIS integral, rough path theory and LDG method, to achieve above goals. 
        
        This paper is organized as follows.  In Section 2, the corresponding memory-dependent PDEEs of the linear SDEs excited by FGN and GWN are derived via the FWIS integral theory, and the corresponding memory-dependent PDEEs of the nonlinear SDEs are derived by utilizing the rough path theory. The LDG method is introduced in Section 3, which is capable of obtaining high-precision transient PDF solutions of the memory-dependent PDEEs. In Section 4, the accuracy of the proposed numerical method and the correctness of the memory-dependent PDEEs are verified through some numerical examples. Finally, some discussions and conclusions are drawn in Section 5.

\section{The governing equations}
    
	\subsection{The mathematical basis for Langevin equation excited by FGN and GWN}
	
	We start with the following Langevin equation excited by combined FGN and GWN
	\begin{equation}\label{mSDS}
	\dot{X}_t= f(t,X_t)+g(t,X_t)\xi_t+h(t,X_t)\xi_t^H,
	\end{equation}
	where $X_0=x_0$, the initial value of $X_t$ at time $t=0$, is given and deterministic. $ f, g$ and $ h $ are coefficients of drift and diffusion terms, respectively. $\xi_t $ is a unit GWN and $\xi^{H}_t$ is a unit FGN with Hurst parameter $1/2<H<1$. $\xi_t$ and $\xi^{H}_t$ are independent. 
        The autocorrelation function \cite{qian2003fractional} of $\xi^{H}_t$ with $\tau$ as the correlation time is
		\begin{align}\label{fgn-cor}
		R_H(\tau)=\mathbb{E}[\xi^H_0\xi^H_\tau]=2H(2H-1)|\tau|^{2H-2}+2H|\tau|^{2H-1}\delta(\tau).
		\end{align}
For $ H = 1/2 $, $\xi^{H}_t$ reduces to $\xi_t$, where the first term in (\ref{fgn-cor}) is zero and the second term is Dirac-Delta function $ \delta(\tau) $. From (\ref{fgn-cor}), in contrast to $\xi_t$, $\xi^{H}_t$ exhibits the strong long-range correlations and non-Markovian characteristics. For $ 1/2 < H < 1 $, the spectral density of $\xi^{H}_t$ can be calculated by
		\begin{align*}
		S_H(\omega)=\frac{1}{2\pi}\int_{-\infty}^{\infty}R(\tau)e^{-i\omega \tau}\mathrm{d}\tau=\frac{\Gamma(2H+1)\sin(H\pi)}{\pi}|\omega|^{1-2H},
		\end{align*}
		which corresponds to the well known $ 1/\omega^\alpha $ noise \cite{weissman19881}. Therefore, $\xi^{H}_t$ has a simple power-law spectral density function over all frequencies $ \omega $.

	From now, we are going to give a quick (non-rigorous) introduction into the topic of Langevin equation in the framework of SDEs and explore their connections to the Fokker-Planck type equations. Before we can discuss the theory of SDEs, we have to introduce the fractional Brownian motion (FBM), which is of fundamental importance for this subject. The FBM ($B^H_t, t\geq0$) is a zero mean centered Gaussian process
		 $\mathbb{E}[B^H_t]=0,$
		with $B^H_0=0$ and auto-covariance function
		$$\mathbb{E}[B^H_tB^H_s]=1/2 (t^{2H}+s^{2H}-|t-s|^{2H}),\qquad t,s\geq 0.$$
		Here $ H $ is the Hurst index, and has a bounded range between 0 and 1. For $ H=1/2 $ the correlations vanish, and FBM reduces to BM $(W_t, t\geq 0)$. The increment processes $ B^H_t-B^H_s $ have a Gaussian distribution with 
		\begin{equation*}
		\mathbb{E}[B^H_t-B^H_s]=0, \qquad
		\mathbb{E}[(B^H_t-B^H_s)^2]=|t-s|^{2H},\qquad t,s\geq 0,
		\end{equation*}
and stationary increment property.
$ \xi^H_t $ can be formally written as
	$\xi^H_t = \frac{\mathrm{d}B^H_t}{\mathrm{d}t},$
		which is Gaussian and stationary and  $\xi_t$ can be formally written as
	$\xi_t = \frac{\mathrm{d}W_t}{\mathrm{d}t}$. The FBM is not a semi-martingale unless $H =1/2$, so the usual It\^o's stochastic calculus is not valid. Nevertheless, there are several approaches to a stochastic calculus in order to interpret (\ref{mSDS}) in a meaningful way. We do not discuss in this paper these approaches, but refer the interested reader to \cite{biagini2008stochastic,mishura2008stochastic,pei2024almost,alos2003stochastic,coutin2002stochastic,carmona2003stochastic,zahle1998integration,ducan2000stochastic,dai1996ito,lin1995stochastic}. 
	
	In this paper, we understand (\ref{mSDS}) in the sense of Stratonovich type
	\begin{equation}\label{mSDEs}
	\mathrm{d}X_t=f(t,X_t) \mathrm{d}s+g(t,X_t) \circ \mathrm{d}W_t+h(t,X_t) \circ \mathrm{d} B^H_t,
	\end{equation}
	where $\circ \mathrm{d} W_t$ indicates Stratonovich integral in usual case and $\circ \mathrm{d} B^H_t$ indicates symmetric pathwise integral in the sense of \cite{dai1996ito,zahle1998integration,lin1995stochastic}.

	\subsection{The memory-dependent PDEEs associated with the linear autonomous SDEs}

    Generally speaking, for the autonomous SDEs (\ref{mSDEs}) driven purely by BM, the corresponding FPK equation is well-established \cite{sun2006stochastic} as
	\begin{align}\label{FPK-Stro}
	\frac{\partial }{\partial t}p(x,t)=-\frac{\partial }{\partial x}\big\{\big(f(t,x)+\frac{1}{2}g(t,x)\frac{\partial g(t,x)}{\partial x}\big)p(x,t)\big\}
	+\frac{\partial^2 }{\partial x^2}\big\{\frac{1}{2}g^2(t,x)p(x,t)\big\}.
	\end{align}
	Conversely, for the autonomous SDEs (\ref{mSDEs}) driven purely by FBM without a drift term, the corresponding governing equation reads \cite{choi2021entropy}
\begin{align}\label{FPK-onlyfbm}
	\frac{\partial }{\partial t}p(x,t)=-\frac{\partial }{\partial x}\big\{Ht^{2H-1}h'(x)h(x)p(x,t)\big\}
	+\frac{\partial^2 }{\partial x^2}\big\{Ht^{2H-1}h^2(x)p(x,t)\big\}.
	\end{align}
	
	However, the governing equations associated with the linear and nonlinear autonomous SDEs (\ref{mSDEs}) excited by FGN and GWN have not yet been determined. We plan to develop a new method (non-Markovian PDEM), to derive the memory-dependent PDEEs for PDFs of SDEs excited by FGN and GWN.
    
Now, we first consider the autonomous SDEs (\ref{mSDEs}) in the linear case, i.e., $f(t, X_t)= A_t X_t, g(t, X_t)= B_t X_t$ and $ h(t, X_t)= C_t X_t$, where $ A_t,B_t,C_t $ are continuous functions w.r.t. $ t $. We rewrite (\ref{mSDEs}) as following
	\begin{equation}\label{mSDEline}
	\mathrm{d}X_t=A_t X_t \mathrm{d}t+B_t X_t \circ \mathrm{d}W_t+C_t X_t  \circ \mathrm{d} B^H_t.
	\end{equation}
	
	In this subsection, we derive the memory-dependent PDEEs for the linear SDEs (\ref{mSDEline}) by the FWIS integral theory \cite{ducan2000stochastic,hu2005stochastic}.
    The key advantage of this method is that one can transform the symmetric pathwise integral $\int X \circ \mathrm{d} B^H$ into the FWIS integral $\int X \diamond \mathrm{d} B^H$, thereby changing the expectation of random terms from non-zero to zero, i.e., $\mathbb{E}[\int X  \circ \mathrm{d} B^H ]\ne0$ but $\mathbb{E}[\int X \diamond\mathrm{d}B^H]=0 $. This property is crucial to derive the memory-dependent PDEEs.
	
\begin{definition}[{\cite[Definition 3.4 and Definition 3.5]{ducan2000stochastic}}]
	Let $\mathscr{L}^2_{\phi}(0, T)$ be be the space of measurable, scalar-valued stochastic processes $(F_t,t\in[0,T])$, such that
	$$ \mathbb{E}\bigg[\int_0^T\int_0^T F_s F_t \phi(s,t)\mathrm{d}s\mathrm{d}t\bigg]<\infty, $$ 
	where $\phi(t,s)=H(2H-1)|t-s|^{2H-2}$.
\end{definition}

\begin{lemma}[{\cite[Theorem 3.12]{ducan2000stochastic}}]\label{sym-path-to-fwick}
	If $ F\in \mathscr{L}^2_{\phi}(0, T)$, then the symmetric pathwise integral $\int_{0}^{t} F_s  \circ \mathrm{d} B^H_s$ and the FWIS integral $\int_{0}^{t} F_s \diamond\mathrm{d}B^H_s$ exist, and the following equality is satisfied:
	\begin{equation}\label{symtofwick}
	\int_{0}^{t} F_s  \circ \mathrm{d} B^H_s= \int_{0}^{t} F_s \diamond\mathrm{d}B^H_s+\int_{0}^{t}D^{\phi}_sF_s\mathrm{d}s, \,  \forall t\in[0,T], \,a.s.,
	\end{equation}
	where $ D^{\phi}_t F_t $ is the Mallivian derivative of $F_t$ {\rm \cite{ducan2000stochastic}}.
\end{lemma}

\begin{remark}
	$\int_{0}^{t}D^{\phi}_sF_s\mathrm{d}s$ in Lemma \ref{sym-path-to-fwick} is a correction term between the symmetric pathwise integral and the FWIS integral w.r.t. FBM term, just as Wong-Zakai correction. However, it is not easy to obtain the analytical form except some linear cases, for instance, the following Lemma \ref{Malianx}.
\end{remark}

\begin{lemma}[{\cite[Theorem 4.2]{ducan2000stochastic}}]\label{mal-to-fwick}
	If $ F\in \mathscr{L}^2_{\phi}(0, T)$ and $ \sup_{0\leq t\leq T}\mathbb{E}[|D^{\phi}_tF_t|^2]<\infty $, then for $ s,t\in[0,T] $, one has
	\begin{equation}\label{maltofbm}
	D^{\phi}_s\Big(\int_{0}^{t} F_u \diamond \mathrm{d} B^H_u\Big)= \int_{0}^{t} D^{\phi}_sF_u \diamond \mathrm{d} B^H_u+\int_{0}^{t}F_u\phi(s,u)\mathrm{d}u, \,  \forall s,t\in[0,T], \,a.s.
	\end{equation}
\end{lemma}

\begin{lemma}\label{Malianx}
	If $ A_t, B_t $ and $ C_t $ are measurable and essentially bounded deterministic functions in $ t $, then, Eq. (\ref{mSDEline}) admits a unique solution. Moreover, let $X_t$ be the solution of (\ref{mSDEline}), then
	\begin{equation}\label{Mallinex}
	D^{\phi}_tX_t= X_t\int_{0}^{t}\phi(t,s)C_s\mathrm{d}s, \, \forall t\in[0,T], \, a.s.
	\end{equation}
\end{lemma}
\begin{proof}
	A proof is provided in \ref{app-Malianx}.
\end{proof}
\begin{lemma}\label{itolinax}Let $ X_t $ be the solution of (\ref{mSDEline}). Suppose that $ F(t,x) $ is any twice differentiable function of two variables. If $ A_t, B_t $ and $ C_t $ satisfy the conditions of Lemma \ref{Malianx}, then
	\begin{align}\label{itoline}
	\mathrm{d}F(t,X_t)=&\Big(\frac{\partial F}{\partial t}(t,X_t)+\big(A_tX_t+\frac{1}{2}B^2_tX_t+\hat{C}_tX_t\big)\frac{\partial F}{\partial x}(t,X_t) \cr
	&+\big(\frac{1}{2}B^2_tX_t^2+\hat{C}_tX_t^2\big)\frac{\partial^2 F}{\partial x^2}(t,X_t)\Big)\mathrm{d}t+B_tX_t\frac{\partial F}{\partial x}(t,X_t)\mathrm{d}W_t\cr
	&+C_tX_t\frac{\partial F}{\partial x}(t,X_t) \diamond\mathrm{d}B^H_t,
	\end{align}
	where $ \hat{C}_t= C_t\int_{0}^{t}\phi(t,r)C_r\mathrm{d}r$,
	$\mathrm{d}W$ indicates the usual It\^o integral and $\diamond\mathrm{d}B^H$ denotes the FWIS integral.
\end{lemma}
\begin{proof}
	A proof is provided in \ref{app-itolinax}.
\end{proof}
	
	\begin{theorem}\label{fplinax}
		Let $ X_t $ be the solution of (\ref{mSDEline}). If $ A_t, B_t $ and $ C_t $ satisfy the conditions of Lemma \ref{Malianx}. Then the memory-dependent PDEEs for PDFs $ p(x,t) $ of the solutions $ X_t $, satisfies the following equation
		\begin{align}\label{FPKline}
		\frac{\partial }{\partial t}p(x,t)=&-\frac{\partial }{\partial x}\big\{\big(A_tx+\frac{1}{2}B^2_tx+\hat{C}_tx\big)p(x,t)\big\}+\frac{\partial^2 }{\partial x^2}\big\{\big(\frac{1}{2}B^2_tx^2+\hat{C}_tx^2\big)p(x,t)\big\},
		\end{align}
		where $ \hat{C}_t= C_t\int_{0}^{t}\phi(t,r)C_r\mathrm{d}r$ and the initial condition is given by $p(x,0)=\delta(x-x_0)$.
	\end{theorem}
	\begin{proof}
		A proof is provided in \ref{app-fplinax}.
	\end{proof}

	\begin{remark}\label{fplinax-multi}
		For simplicity, we only consider one-dimensional case in this paper, the conclusion can be generalized into higher dimensional cases, whose coefficients $ A_t, B_t, C_t $ should satisfy the commutativity conditions, that is, $ A_tB_t=B_tA_t, B_tC_t=C_tB_t, C_tA_t=A_tC_t$ where $ A_t, B_t, C_t $ are all time-dependent matrixes.
	\end{remark}

	\subsection{The memory-dependent PDEEs associated with the nonlinear SDEs}
	We now consider the nonlinear SDEs with slightly more general coefficients, and rewrite (\ref{mSDEs}) as following equivalent form
	\begin{equation}\label{mSDEscom}
	\mathrm{d}X_t=f(X_t)\mathrm{d}t+g(X_t)\circ \mathrm{d}W_t+h(X_t) \circ \mathrm{d}B^H_t.
	\end{equation}
	
	Unfortunately, it is currently not feasible to compute the analytic form of the Malliavin derivative \( D^{\phi}_tX_t \) for the nonlinear SDEs (\ref{mSDEscom}), and thus the corresponding memory-dependent PDEEs cannot be obtained using the method described in the previous subsection. Consequently, in this subsection, we will employ the rough path theory as an alternative framework to derive the memory-dependent PDEEs.
	
	The rough differential equations (RDEs) corresponding with SDEs (\ref{mSDEscom}) are 
	\begin{equation}\label{rde1}
	\mathrm{d}X_t^{\mathrm{R}}=f^{\mathrm{R}}(X_t^{\mathrm{R}}) \mathrm{d} t+ g^{\mathrm{R}}(X_t^{\mathrm{R}}) \mathrm{d}^{\mathrm{R}} W_t+ h^{\mathrm{R}} (X_t^{\mathrm{R}}) \mathrm{d}^{\mathrm{R}} B_t^H,
	\end{equation}
	where the integrals are understood in the rough path sense. If $ 1/2<H<1$, the integral is understood as Young's integration \cite{young1936inequality,zahle1998integration}; if $1/4<H\leq1/2$, it is understood in the rough path sense of Lyons \cite{lyons1998differential,coutin2002stochastic}, for instance, $\mathrm{d}^{\mathrm{R}}B^{1/2}$, here $B^{1/2}=W$.
	
	Next, we will show that $X$ are also the solutions to the RDEs (\ref{rde1}) with the coefficients
	\begin{eqnarray}\label{coeff}
	f^{\mathrm{R}}(\cdot)=f(\cdot), g^{\mathrm{R}}(\cdot)=g(\cdot), h^{\mathrm{R}}(\cdot)=h(\cdot).
	\end{eqnarray}
\begin{lemma}\label{sde-to-rde}
		Let $f, g, h $ are twice-differentiable and bounded functions, and suppose that $X=(X_t)_{t \in[0, T]}$ are the solutions to the SDEs (\ref{mSDEscom}),
		then $X$ are the solutions to the RDEs (\ref{rde1}) with coefficients given by (\ref{coeff}).
		In other words, the solutions $X^{\mathrm{R}}$ of (\ref{rde1}) and $X$ of (\ref{mSDEscom}) coincide $\mathbb{P}$-almost surely, that is, we have
		$$
		\mathbb{P}\left(X_t=X_t^{\mathrm{R}}, t \in[0, T]\right)=1 .
		$$
	\end{lemma}
	\begin{proof}
		A proof is provided in \ref{app-sde-to-rde}.
	\end{proof}
	
Using Lemma \ref{sde-to-rde}, we can reformulate the SDEs (\ref{mSDEscom}) in the framework of rough path theory with the following RDEs
\begin{equation}\label{RDEvec}
\mathrm{d}X_t^{\mathrm{R}}= V_0(X^{\mathrm{R}}_t)\mathrm{d}t+\sum_{i=1}^{d}V_i(X_t^{\mathrm{R}})\mathrm{d}^{\mathrm{R}}B^{H,i}_t,
\end{equation}
where the $ V_i $'s are $ C^{\infty} $ bounded vector fields on $ \mathbb{R}^n $ and $(B^{H,i})_{i=0,1,\cdots,d} $ are one-dimensional FBMs with Hurst index $ H \geq 1/2$ for $ i=1,\cdots,d $. 
\begin{definition}
	A smooth vector field $ V:\mathbb{R}^n\rightarrow\mathbb{R}^n $ is simply a smooth map $ x\mapsto (v_1(x),\cdots,v_n(x)) $, defined a differential operator acting on the smooth function $ f:\mathbb{R}^n\rightarrow\mathbb{R} $ as 
	$$  (Vf)(x)=\sum_{i=1}^{n}v_i(x)\frac{\partial f}{\partial x_i}.  $$ 
	Assume that the Lie brackets $ [V_i, V_j] = V_iV_j-V_jV_i, 0 \leq i, j \leq d $, all vanish.
\end{definition}

\begin{remark}
	As for the RDEs (\ref{rde1}), it is the special case of RDEs (\ref{RDEvec}) by taking $n=1, d=2, V_0(\cdot)=f(\cdot), V_1(\cdot)=g(\cdot), B^{H,1}=W, V_2(\cdot)=h(\cdot), B^{H,2}=B^H$ for $t\in [0,T]$.
\end{remark}

For $ i = 0,1,\dots, d $, let us denote by $ (e^{tV_i})_{t\in\mathbb{R}} $ the (deterministic) flow associated with the ordinary differential equations
$$ \frac{\mathrm{d}x_t}{\mathrm{d}t}=V_i(x_t). $$
\begin{lemma}\label{flowuniq}
	The flow $ \Phi_t $ associated with the RDEs (\ref{RDEvec}) is given by the formula
	\begin{equation*}
	\Phi_t=e^{tV_0}\circ e^{B_t^{H,1}V_1}\circ \cdots \circ e^{B_t^{H,d}V_d},
	\end{equation*}
	where $ \phi \circ \psi $ is the composite of two maps $ \phi$, $\psi $ of $ \mathbb{R}^n $ into itself.
\end{lemma}
	
	\begin{theorem}\label{fpcomm}
		Let $ X=(X_t)_{t\in[0,T]} $ be the solutions of the SDEs (\ref{mSDEscom}), if $ f, g$ and $ h$ are bounded smooth functions and satisfied the commutativity conditions. Then the memory-dependent PDEE for PDFs $ p(x,t) $ of the solutions $ X_t $, satisfies the following equation
		\begin{align}\label{FPKcase33}
		\frac{\partial }{\partial t}p(x,t)=&-\frac{\partial }{\partial x}\big\{\big(f(x)+\frac{1}{2}g(x)g'(x)+Ht^{2H-1}h(x)h'(x)\big)p(x,t)\big\}\cr
		&+\frac{\partial^2 }{\partial x^2}\big\{\big(\frac{1}{2}g^2(x)+Ht^{2H-1}h^2(x)\big)p(x,t)\big\},
		\end{align}
		where the initial condition is given by $p(x,0)=\delta(x-x_0)$.
	\end{theorem}
	\begin{proof}
		A proof is provided in \ref{app-fpcomm}.
	\end{proof}

	If $ f(x)=g(x)\equiv0 $, consider
	\begin{equation}\label{SDEscomxh}
	\mathrm{d}X_t= h(X_t)\circ \mathrm{d}B^H_t,\qquad X_0=x_0,
	\end{equation}
	then we have the following corollary.
	\begin{corollary}\label{fpcomm-onlyh}
		Let $ X=(X_t)_{t\in[0,T]} $ be the solutions of (\ref{SDEscomxh}), if $ h(x) $ is $ C^{\infty} $ bounded function. Then the memory-dependent PDEE for PDFs $ p(x,t) $ of the solutions $ X_t $, satisfies the following equation
		\begin{align}\label{FPKcase33h}
		\frac{\partial }{\partial t}p(x,t)=-\frac{\partial }{\partial x}\big\{Ht^{2H-1}h'(x)h(x)p(x,t)\big\}
		+\frac{\partial^2 }{\partial x^2}\big\{Ht^{2H-1}h^2(x)p(x,t)\big\},
		\end{align}
		which has the exact solution
		\begin{align*}
		p(x,t)=C_1\frac{1}{\sqrt{\frac{1}{2}t^{2H}h^2(x)}}\exp\bigg\{-\frac{\hat{x}^2(x)}{2t^{2H}}\bigg\},
		\end{align*}
		where $ C_1 $ is the normalization constant, $ \hat{x}(x)=\int_{x_0}^{x}\frac{1}{h(x)}\mathrm{d}x $, and the initial condition is given by $p(x,0)=\delta(x-x_0)$.
	\end{corollary}
	\begin{remark}
		When $A_t, B_t $ and $C_t $ are constants in the linear SDE (\ref{mSDEline}), which is a special case of (\ref{mSDEscom}).
	\end{remark}
	
	\section{The LDG method for memory-dependent PDEEs}
	In this section, we recall the discontinuous finite element space and apply the semi-discrete LDG scheme to solve the transient PDF solutions of the memory-dependent PDEEs.
	
	\subsection{The discontinuous finite element space}
	Let $ a = x_{\frac{1}{2}}< x_{\frac{3}{2}}<\cdots< x_{N+\frac{1}{2}}= b $ be a partition of $ I=[a,b] $. Denote $ I_j=[x_{j-\frac{1}{2}},x_{j+\frac{1}{2}}]$. The mesh size is denoted by $ h_j=x_{j+\frac{1}{2}}-x_{j-\frac{1}{2}}$, with $ h = \max_{1\leq j\leq N} h_j $ being the maximum mesh size. We denote by $ u_{j+\frac{1}{2}}^{-} $ and $ u_{j+\frac{1}{2}}^{+} $ the left and right limits values of the function $ u $ at the discontinuity point $ x_{j+\frac{1}{2}} $, respectively.
	
	The discontinuous finite element space is defined as
	\begin{equation*}
	V_h^k=\big\{v\mid v\in P^k(I_j),x\in I_j,\,j=1,\cdots ,N\big\}.
	\end{equation*}
	where $ P^k(I_j) $ denotes the set of polynomials, defined on the cell $ I_j $, of the degree up to $ k $. Note that functions in $ V_h^k $ might have discontinuities on an element interface.
	
	\subsection{The semi-discrete LDG method}
	In this section, we apply the LDG method for the memory-dependent PDEEs	
    \begin{equation}\label{ldg1}
    	\left\{
	\begin{aligned}
	\frac{\partial }{\partial t}p(x,t)&=-\frac{\partial }{\partial x}D^{(1)}(x,t)p(x,t)+
	\frac{\partial^2 }{\partial x^2}D^{(2)}(x,t)p(x,t),\cr
	p(x,0)&=\delta(x-x_0).
	\end{aligned}
	\right.
	\end{equation}
	
	To establish the LDG method, setting $ v=p_x $, $ w=av $ and $ a=D^{(2)} $, we rewrite the equations~\eqref{ldg1} as the following equivalent first-order systems
	\begin{subnumcases}{}
	\quad \ \ \,\:p_t=w_x+\psi(x,t,p,v),\label{ldg2-a}\\
	\:v(x,t)=p_x(x,t),\label{ldg2-b}\\
	w(x,t)=a(x,t)v(x,t),\label{ldg2-c}\\
	\;\!p(x,0)=\delta(x-x_0), \label{ldg2-d}
	\end{subnumcases}
	where
	\begin{equation*}
	\left\{
	\begin{aligned}
	\psi(x,t,p,v)&=[(D^{(2)}(x,t))_x-D^{(1)}(x,t)]v+[(D^{(2)}(x,t))_{xx}-D^{(1)}(x,t)_x]p,\cr
	a(x,t)&=D^{(2)}(x,t).
	\end{aligned}
	\right.
	\end{equation*}
	
	The LDG method is to find the approximations $ p_h,v_h,w_h $ to the exact solutions $ p,v,w $, respectively, such that $ p_h,v_h,w_h \in V_h^k $. In order to determine the approximate solutions $ p_h,v_h,w_h $,  we firstly multiply (\ref{ldg2-a}), (\ref{ldg2-b}), (\ref{ldg2-c}) and (\ref{ldg2-d}) with arbitrary smooth functions $ r,z,o,q\in V_h^k $, respectively, then integrating over $ I_j $ with $ j = 1,2,\cdots, N $ and integration by parts in (\ref{ldg2-a}) and (\ref{ldg2-b}), we can get the weak form
    \begin{equation}
		\left\{
		\begin{aligned}
	\int_{I_j} p_t(x,t)r(x)\mathrm{d}x&=-\int_{I_j}w(x,t)r_x(x)\mathrm{d}x+\int_{I_j}\psi(x,t,p(x,t),v(x,t))r(x)\mathrm{d}x+[w(\cdot,t)r]\big|_{x_{j-\frac{1}{2}}^{+}}^{x_{j+\frac{1}{2}}^{-}},\cr
	\int_{I_j}v(x,t)z(x)\mathrm{d}x&=-\int_{I_j}p(x,t)z_x(x)\mathrm{d}x+[p(\cdot,t)z]\big|_{x_{j-\frac{1}{2}}^{+}}^{x_{j+\frac{1}{2}}^{-}},\cr
	\int_{I_j}w(x,t)o(x)\mathrm{d}x&=\int_{I_j}a(x,t)v(x,t)o(x)\mathrm{d}x,\cr
	\int_{I_j}u(x,0)q(x)\mathrm{d}x&=\int_{I_j}\delta(x-x_0)q(x)\mathrm{d}x.
	\end{aligned}
	\right.
	\end{equation}
	
	Next, in the above weak formulation, we replace the smooth functions $ r,z,o,q $ with the test functions $ r_h,z_h,o_h,q_h\in V_h^k $, respectively, and the exact functions $ p,v,w $ with the approximations $ p_h,v_h,w_h $. Since the functions in $ V_h^k $ might have discontinuities on an element interface, we must also replace the boundary terms $ p(x_{j+\frac{1}{2}},t) $ and $ w(x_{j+\frac{1}{2}},t) $ with the numerical fluxes $ \hat{p}_{j+\frac{1}{2}}(t) $ and $ \hat{w}_{j+\frac{1}{2}}(t) $, respectively, which are defined by
	\begin{align*}
	\hat{p}_{j+\frac{1}{2}}(t):=p_h(x_{j+\frac{1}{2}}^{-},t),\quad
	\hat{w}_{j+\frac{1}{2}}(t):=w_h(x_{j+\frac{1}{2}}^{+},t),
	\end{align*}
	for $ j = 0,1,\cdots, N $.
	
	Thus, the approximate solutions given by the LDG method is defined as the solutions of the following weak formulation
   \begin{subnumcases}{}
	\int_{I_j}(p_h)_t(x,t)r_h(x)\mathrm{d}x=-\int_{I_j}w_h(x,t)(r_h)_x(x)\mathrm{d}x+\int_{I_j}\psi(x,t,p_h(x,t),v_h(x,t))r_h(x)\mathrm{d}x\notag
	\\
	\qquad\qquad\qquad\qquad\qquad\:\:+\hat{w}_{j+\frac{1}{2}}(t)r_h(x_{j+\frac{1}{2}}^{-})-\hat{w}_{j-\frac{1}{2}}(t)r_h(x_{j-\frac{1}{2}}^{+}),\label{ldg3-a}\\
	\ \ \ \int_{I_j}v_h(x,t)z_h(x)\mathrm{d}x=-\int_{I_j}p_h(x,t)(z_h)_x(x)\mathrm{d}x+\hat{p}_{j+\frac{1}{2}}(t)z_h(x_{j+\frac{1}{2}}^{-})-\hat{p}_{j-\frac{1}{2}}(t)z_h(x_{j-\frac{1}{2}}^{+}),\label{ldg3-b}\\
	\ \,\:	\int_{I_j}w_h(x,t)o_h(x)\mathrm{d}x=\int_{I_j}a(x,t)v_h(x,t)o_h(x)\mathrm{d}x,\label{ldg3-c}\\
	\ \,\,\,
	\int_{I_j}p_h(x,0)q_h(x)\mathrm{d}x=\int_{I_j}\delta(x-x_0)q_h(x)\mathrm{d}x\label{ldg3-d},
	\end{subnumcases}
	for any $ r_h,z_h,o_h,q_h\in V_h^k $.
	
	As a semi-discrete scheme, the algorithm is not difficult for numerical implementation. In fact, given $ p_h $, one firstly uses (\ref{ldg3-b}) to locally solve for $ v_h $, then uses (\ref{ldg3-c}) to locally solve $ w_h $, and finally uses (\ref{ldg3-a}) to locally solve for the update of $ p_h $.

	\section{Numerical results}
	To verify the accuracy and performance of the proposed schemes, we present several numerical experiments. The simulation parameters of the MC are chosen as that the number of sample paths is $ 10^6 $, and the time interval is $ 0.004 $, 
	we adopt $k = 2$ in the spatial discretization \eqref{ldg3-a}-\eqref{ldg3-d} in the LDG simulation.
	
	\subsection{The nonlinear Langevin equation excited by GWN}
	Taking $f(t,x)=ax-bx^3, g(t,x)=\sigma, h(t,x) \equiv 0$ in (\ref{mSDEs}), we understand the nonlinear Langevin equation excited by GWN as a Stratonovich SDE
	\begin{equation}\label{example-1}
	\mathrm{d}X_t= (a X_t-b X_t^3)\mathrm{d}t+\sigma \circ\mathrm{d}W_t,\qquad X_0=x_0.
	\end{equation}
	Then, by (\ref{FPK-Stro}), the corresponding FPK equation of (\ref{example-1}) is
	\begin{equation}\label{fpk-1}
	\frac{\partial }{\partial t}p(x,t)=-\frac{\partial }{\partial x}\{(a x-b x^3)p(x,t)\}+
	\frac{1}{2}\sigma^2\frac{\partial^2 }{\partial x^2}p(x,t),
	\end{equation}
	with the initial condition $p(x, 0) = \delta (x-x_0)$. The exact stationary solution can be obtained 
	\begin{equation}\label{prob-exac-1}
	p(x)=C_1\exp\bigg\{\frac{2a x^2-b x^4}{2\sigma^2}\bigg\},
	\end{equation}
	where $ C_1 $ is the normalization constant. When $ b=0 $, $ a<0 $, the exact transient solution is 
	\begin{equation}\label{prob-exac-2}
	p(x)=\sqrt{\frac{-2a}{2\pi\sigma^2(1-e^{2a t})}}\exp\bigg\{\frac{2a x^2}{2\sigma^2(1-e^{2a t})}\bigg\},
	\end{equation}
	
	In the case of $b\ne0$, the parameters are chosen as $ \Delta t=0.001, R=[-3,3], \Delta x=0.05$ and $a=b=\sigma=1, x_0=0$ to obtain the LDG solution of the FPK equation (\ref{fpk-1}). Since no exact solution exists for (\ref{fpk-1}), MC simulations are conducted to validate the effectiveness of the LDG solution.
	
	Fig. \ref{fig3-3} (a) compares the PDF from MC simulations and the LDG method at different times $ t $. Fig. \ref{fig3-3} (b) shows the stationary PDF using the exact solution (\ref{prob-exac-1}) and LDG solution.
	
	In the case of $b=0$, the parameters are chosen as $ \Delta t=0.001, R=[-6,6], \Delta x=0.05$ and $a=-1, \sigma=1, x_0=0$. The comparison between the exact and the LDG solutions presented in Fig.~\ref{fig3-3} (c). Fig.\ref{fig3-3} (d) shows the stationary PDF.
	
	At different times $t$, which are correspond to different states of nonlinear Langevin equation, we observe excellent agreement between the LDG solutions and the exact solutions.
	
	\begin{figure}[t!]
		\centering
		\setcounter {subfigure} {0} (a){\includegraphics[scale = 0.5]{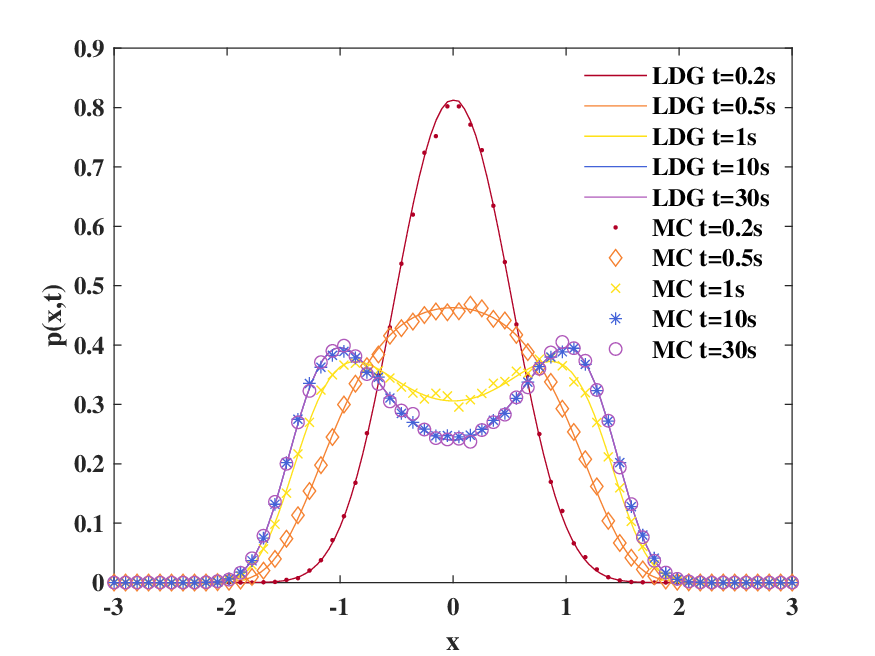}}
		\setcounter {subfigure} {0} (b){\includegraphics[scale = 0.5]{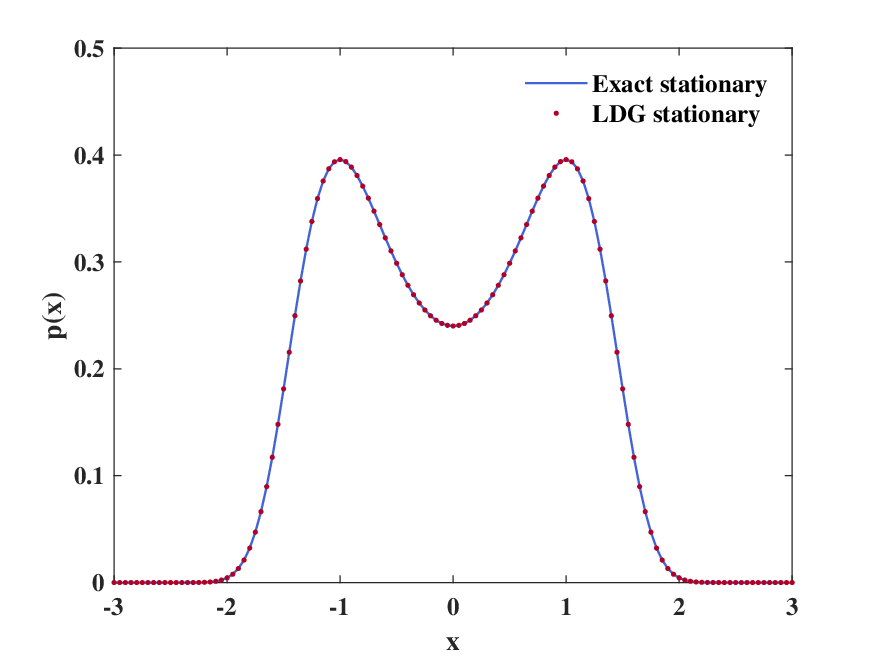}}
		\setcounter {subfigure} {0} (c){\includegraphics[scale = 0.5]{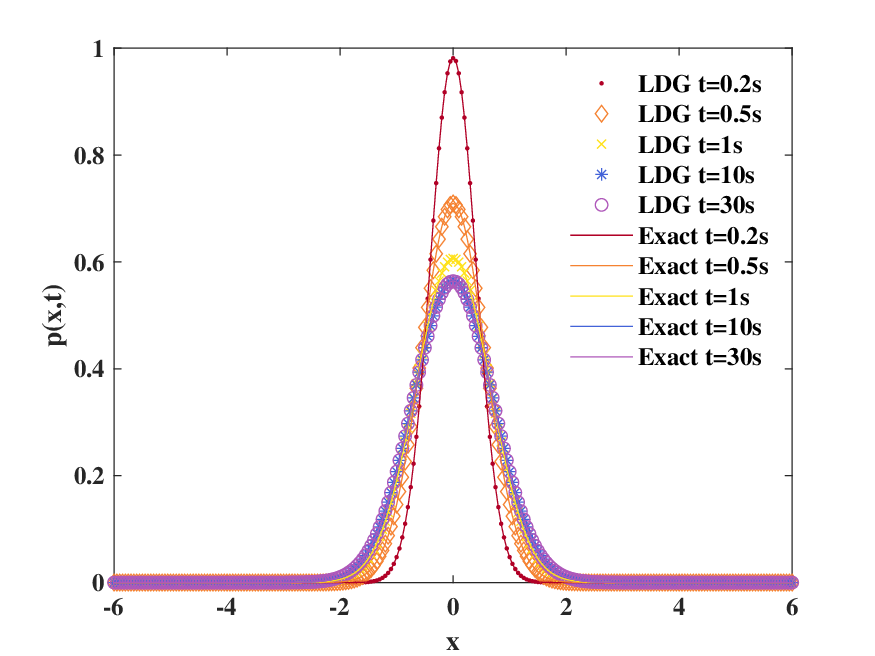}}
		\setcounter {subfigure} {0} (d){\includegraphics[scale = 0.5]{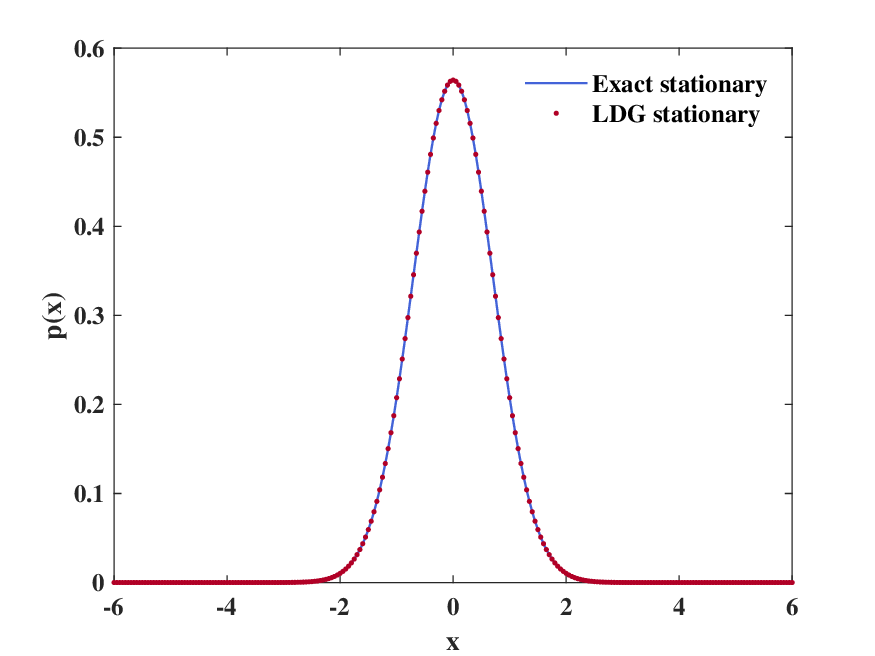}}
		\caption{PDF evolution and stationary PDF of the FPK equation (\ref{prob-exac-1}) where $\sigma=1$ and $a=b=1$ in (a) and (b); $a=-1, b=0$ in (c) and (d).}
		\label{fig3-3}
	\end{figure}
	To further quantify the accuracy of the LDG solutions relative to the exact solution , we compare the error magnitudes obtained using the MC, LDG, FD and PI methods, which is one of the most commonly used techniques for solving the FPK equations for the Markovian SDEs. In this study, FD refers to the first-order central finite difference approximation. We now introduce the $ L^2 $ and $ L^{\infty} $ errors as follows:
	\begin{align}\label{r1}
	L^2= \bigg(\int_{\Omega}\big|\hat{P}(x,t)-\tilde{P}(x,t)\big|^2\mathrm{d}x\bigg)^{1/2},\quad L^{\infty}=\sup_{x\in\Omega}|\hat{P}(x,t)-\tilde{P}(x,t)|,
	\end{align}
	where $ \hat{P} $ denotes the value of the LDG, PI, FD, or MC solutions and $ \tilde{P} $ denotes the value of the exact  solutions.
	\begin{table}[h]
		\caption{ $ L^2 $ error and $ L^{\infty} $ error between the LDG/PI/FD/MC solution and the exact stationary solution of (\ref{fpk-1}) with $ a = b = \sigma = 1, \Delta x=0.05$.}\label{tab1-2}%, \Delta x=0.05,\Delta t=0.05^2
		\begin{tabular*}{\textwidth}{@{\extracolsep\fill}lcccccc}
			\toprule%
			Method  & LDG & PI & FD & MC \\
			\midrule
			$ L^2 $  & 1.4514E-08 & 1.0556E-03  & 1.5274E-03  & 5.8165E-03  \\
			$ L^{\infty}$ & 1.7421E-08  & 9.9435E-04 & 1.3731E-03 & 2.2976E-02 \\
			\bottomrule
		\end{tabular*}
	\end{table}
	
    We compare the LDG, PI, FD, and MC solutions against the exact stationary solution (\ref{prob-exac-1}) and present the $ L^2 $ and $ L^{\infty} $ errors with the same parameter $\Delta x=0.05$ in Table \ref{tab1-2} and $\Delta x=0.05,\Delta t=0.001$ in Table \ref{tab1}. 
    
In Table \ref{tab1-2}, the results demonstrate that the LDG method achieves higher precision compared to other methods for solving the stationary solution of the FPK equation (\ref{fpk-1}).
    
	\begin{table}[h]
		\caption{$ L^2 $ error and $ L^{\infty} $ error between the LDG/PI/FD solutions and the exact solutions of (\ref{fpk-1}) with $ a = -1, b = 0, \sigma = 1, \Delta x=0.05,\Delta t=0.001 $.}\label{tab1}%, \Delta x=0.05,\Delta t=0.001
		\begin{tabular*}{\textwidth}{@{\extracolsep\fill}lcccccc}
			\toprule%
			& \multicolumn{5}{@{}c@{}}{$ L^2 $ error} \\\cmidrule{2-6}%
			Method & $t=0.2$ s  & $t=0.5$ s& $t=1$ s& $t=10$ s& $t=30$ s\\
			\midrule
			LDG  & 1.5809E-04 & 1.0023E-04 & 5.4830E-05 &  4.1307E-10 & 4.1281E-10  \\
			PI  & 3.2464E-03 & 2.0672E-03 & 1.2952E-03 & 6.7113E-04 & 6.7113E-04    \\
			FD  & 1.4866E-03 & 7.2319E-04 & 5.6137E-04 & 5.4947E-04 & 5.4947E-04    \\
			\midrule
			& \multicolumn{5}{@{}c@{}}{$ L^{\infty} $ error} \\\cmidrule{2-6}%
			Method & $t=0.2$ s & $t=0.5$ s& $t=1$ s& $t=10$ s& $t=30$ s\\
			\midrule
			LDG  & 1.6339E-04 & 7.4697E-05 & 4.3296E-05 & 3.5873E-10 & 3.5377E-10  \\
			PI  & 4.4376E-03 & 2.3956E-03 & 1.3866E-03 & 6.9258E-04 & 6.9258E-04    \\
			FD  & 2.1248E-03 & 8.6159E-04 & 5.8607E-04 & 5.2962E-04 & 5.2962E-04    \\
			\bottomrule
		\end{tabular*}
	\end{table}
    
    In Table \ref{tab1}, the $ L^2 $ and $ L^{\infty} $ errors at different times $t$ are presented. As the nonlinear Langevin equation approaches steady state (e.g., at $t=10$ s and $t=30$ s), the improved regularity of the equation results in significantly better accuracy for the LDG  solutions, surpassing other numerical methods by up to six orders of magnitude. Conversely, even during transient states when the FPK equation exhibits limited regularity, the LDG method maintains higher precision compared to the PI and FD methods. These results confirm the efficiency of the LDG method.
    
    To proceed, taking $f(t,x)\equiv 0, g(t,x)=\sqrt{a}t^bx, h(t,x) \equiv 0$ in (\ref{mSDEs}), we understand the nonlinear Langevin equation excited by GWN as a Stratonovich SDE
	\begin{equation}\label{example1-3}
	\mathrm{d}X_t= \sqrt{a}t^{b}X_t \circ\mathrm{d}W_t,\qquad X_0=x_0.
	\end{equation}
	By (\ref{FPK-Stro}), the corresponding FPK equation of (\ref{example1-3})
	\begin{equation}\label{ldg1133}
	\frac{\partial }{\partial t}p(x,t)=-\frac{\partial }{\partial x}\big\{\frac{1}{2}at^{2b}xp(x,t)\big\}+
	\frac{\partial^2 }{\partial x^2}\big\{\frac{1}{2}at^{2b}x^2p(x,t)\big\},
	\end{equation}
	with the initial condition $ p (x, 0) = \delta (x-x_0) $ and the exact solution of (\ref{ldg1133}) is
	\begin{equation}\label{exact1133}
	p(x,t)=\sqrt{\frac{2b+1}{2\pi at^{2b+1}x^2}}\exp\bigg\{-\frac{(2b+1)(\ln x-\ln x_0)^2}{2at^{2b+1}}\bigg\}.
	\end{equation}
	
	For the numerical calculation of the LDG  solutions of the FPK equation (\ref{ldg1133}), we select the parameters $ \Delta x=0.067, R=[0,8], \Delta t=0.00167 $ and $a=0.02, b=0.3, x_0=2$. And the FD method employs the same parameters.
	
	The LDG solutions at different times $t$ are compared with the exact solutions in Fig. \ref{fig3-4} (a). The absolute errors between the exact and LDG solutions are shown in Fig. \ref{fig3-4} (b). Figures \ref{fig3-4} (c) and (d) display the PDF evolution surfaces for the LDG and exact solutions, respectively. From Fig. \ref{fig3-4} (b), it is evident that the numerical scheme performs well across most regions.
	
	\begin{figure}[t!]
		\centering
		\setcounter {subfigure} {0} (a){\includegraphics[scale = 0.50]{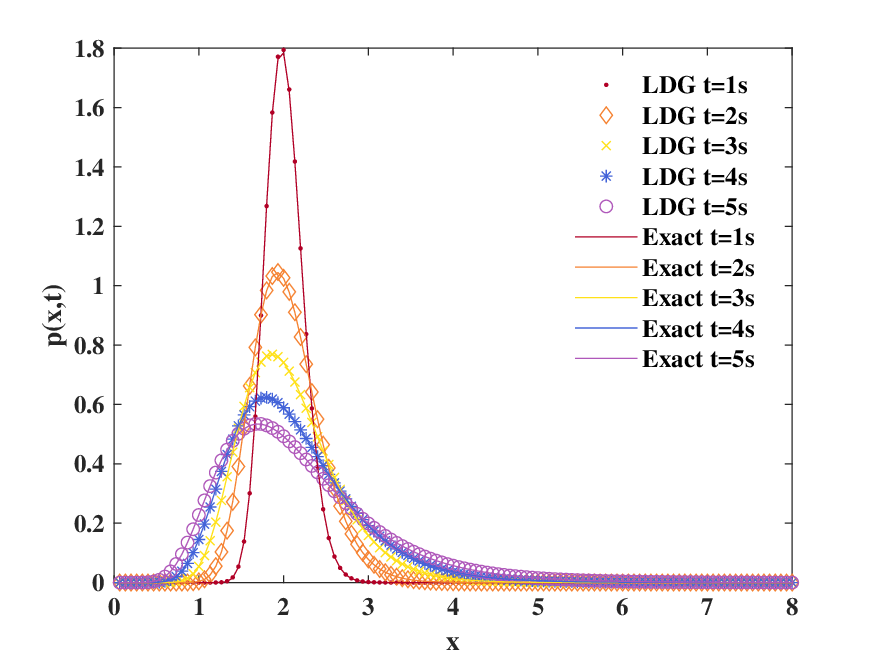}}
		\setcounter {subfigure} {0} (b){\includegraphics[scale = 0.50]{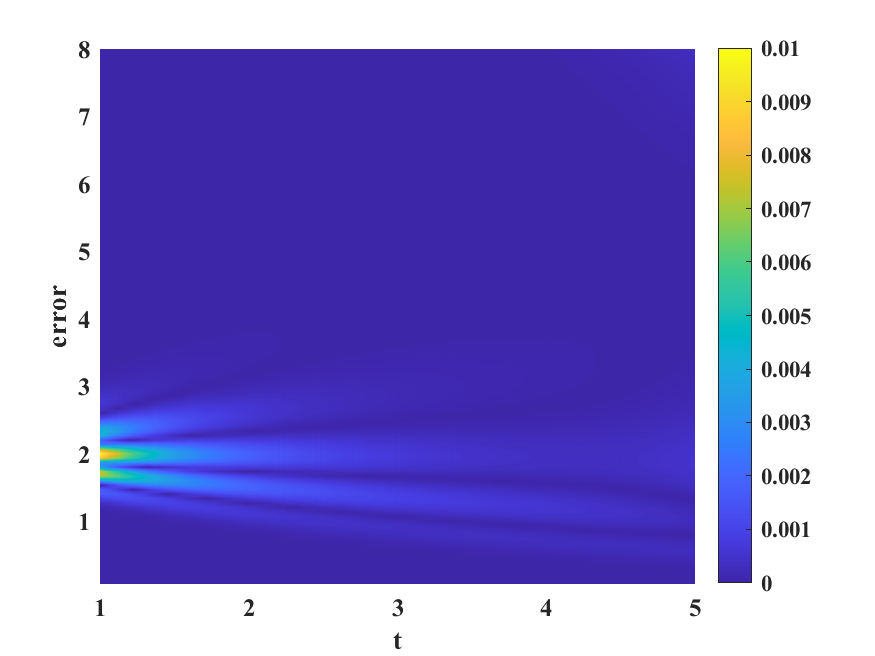}}
		\setcounter {subfigure} {0} (c){\includegraphics[scale = 0.50]{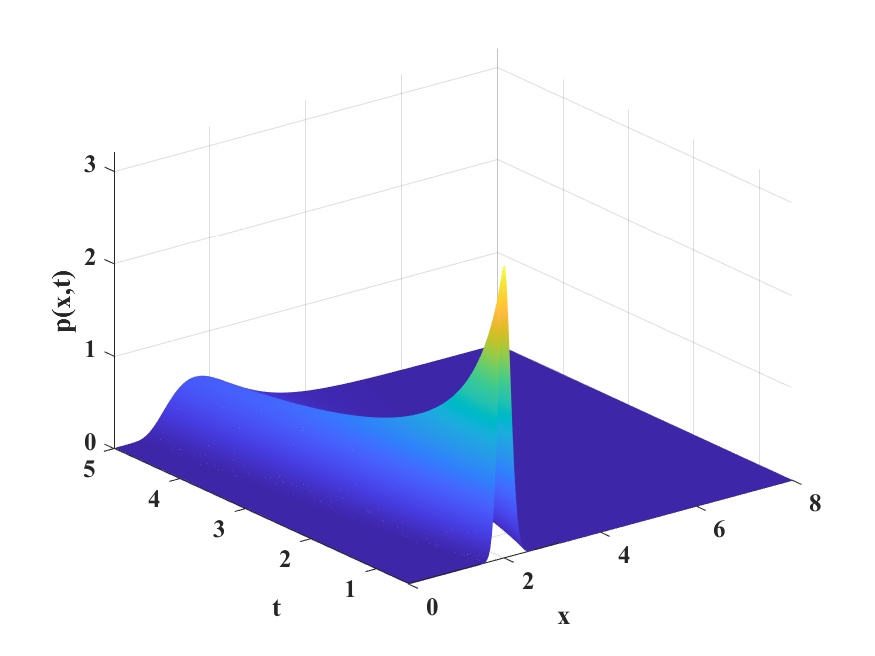}}
		\setcounter {subfigure} {0} (d){\includegraphics[scale = 0.50]{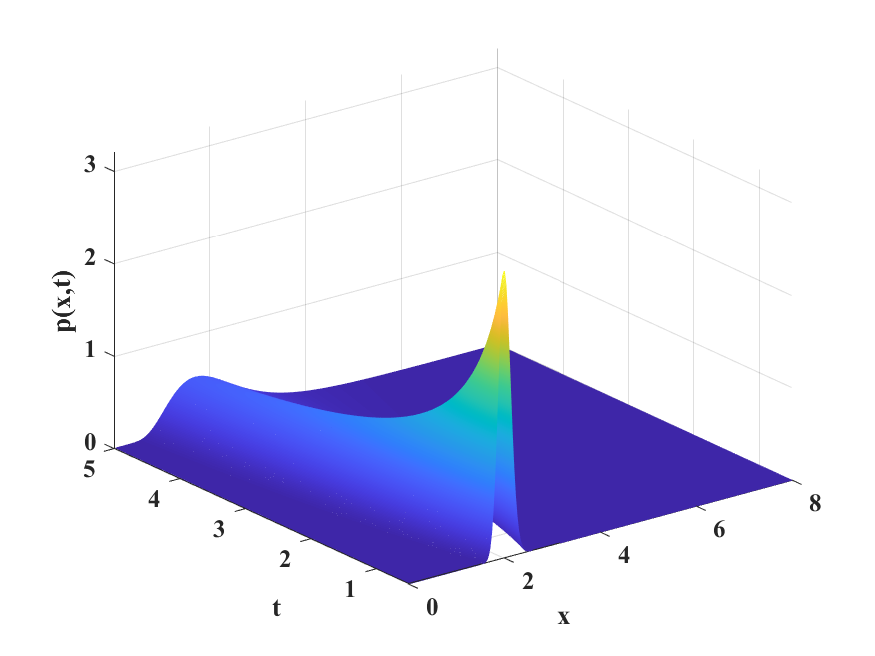}}
		\caption{PDF evolution (a) of the FPK equation (\ref{ldg1133}), the absolute errors (b) between the exact solutions and the LDG solutions, and PDF evolution surface of LDG solutions (c) and the exact solutions (d) with $a=0.02, b=0.3, x_0=2$.}
		\label{fig3-4}
	\end{figure}

	In Table \ref{tab2}, the $ L^2 $ and $ L^\infty $ errors at different times $t$ are presented. The results indicate that the LDG method provides significantly more accurate transient solutions compared to the FD method, validating its effectiveness for solving time-dependent FPK equations.
	
	\begin{table}[h]
		\caption{$ L^2 $ error and $ L^{\infty} $ error between the LDG/FD solutions and the exact solutions of (\ref{ldg1133}) with $ a = 0.02, b = 0.3, \Delta x=0.067,\Delta t=0.00167$.}\label{tab2}%, \Delta x=0.067,\Delta t=0.00167
		\begin{tabular*}{\textwidth}{@{\extracolsep\fill}lcccccc}
			\toprule%
			& \multicolumn{5}{@{}c@{}}{$ L^2 $ error} \\\cmidrule{2-6}%
			Method & $t=1$ s & $t=2$ s& $t=3$ s& $t=4$ s& $t=5$ s\\
			\midrule
			LDG  & 5.4249E-03 & 1.4214E-03 & 6.9663E-04 &  4.4969E-04 & 5.6859E-04  \\
			FD  & 1.6890E-02 & 4.5168E-03 & 2.2446E-03 &  1.4485E-03 & 1.2037E-03    \\
			\midrule
			& \multicolumn{5}{@{}c@{}}{$ L^{\infty} $ error} \\\cmidrule{2-6}%
			Method & $t=1$ s & $t=2$ s& $t=3$ s& $t=4$ s& $t=5$ s\\
			\midrule
			LDG & 1.6890E-03 & 1.9306E-03 & 9.3514E-04 &  5.7940E-04 & 5.2386E-04   \\
			FD  & 3.1451E-02 & 6.0668E-03 & 2.9421E-03 & 1.8503E-03 & 1.2410E-03    \\
			\bottomrule
		\end{tabular*}
	\end{table}

	\subsection{The linear Langevin equation excited by combined FGN and GWN}
	In this subsection, we first consider the linear Langevin equation excited by combined multiplicative FGN and GWN. Specifically, the functions are defined as $f(t,x)=ax, g(t,x)=bx$ and $ h(t,x)=cx$. Consequently, the linear SDEs (\ref{mSDEline}) can be expressed as follows
	\begin{equation}\label{example1-9}
	\mathrm{d}X_t= aX_t\mathrm{d}t+bX_t\circ\mathrm{d}W_t+cX_t\circ\mathrm{d}B^H_t,\qquad X_0=x_0,
	\end{equation}
	where $ a, b $ and $ c $ are constants. By Theorem \ref{fplinax}, the PDFs of (\ref{example1-9}) satify
 the following memory-dependent PDEE
	\begin{align}\label{ldg1139}
	\frac{\partial }{\partial t}p(x,t)=&-\frac{\partial }{\partial x}\big\{\big(ax+\frac{1}{2}b^2x+Ht^{2H-1}c^2x\big)p(x,t)\big\}
	\cr
	&+\frac{\partial^2 }{\partial x^2}\big\{\big(\frac{1}{2}b^2x^2
	+Ht^{2H-1}c^2x^2\big)p(x,t)\big\},
	\end{align}
	with the initial condition $ p (x, 0) = \delta (x-x_0) $. When $ b=0 $, the exact solution is
	\begin{equation*}
	p(x,t)=\frac{1}{x\sqrt{2\pi c^2t^{2H}}}\exp\bigg\{-\frac{\big(\ln(x/x_0)-x_0-at\big)^2}{2c^2t^{2H}}\bigg\}.
	\end{equation*}
	
	We select the parameters $ \Delta x=0.1, R=[0,6], \Delta t=0.0005  $ and $ a=-0.5, b=0.25, c=0.25, H=0.8, x_0=2$ to obtain the LDG  solutions of Eq. (\ref{ldg1139}) and compared those with MC solutions in Fig. \ref{fig3-9} (a).
	
	For the case where $b=0$, we choose the parameters $ \Delta x=0.05, R=[0,6], \Delta t=0.0001 $ and $ a=-0.5, c=0.5 $ to demonstrate that the LDG solutions closely match the exact solutions of Eq. (\ref{ldg1139}), as shown in Fig. \ref{fig3-9} (b).
	
	\begin{figure}[t!]
		\centering
		\setcounter {subfigure} {0} (a){\includegraphics[scale = 0.5]{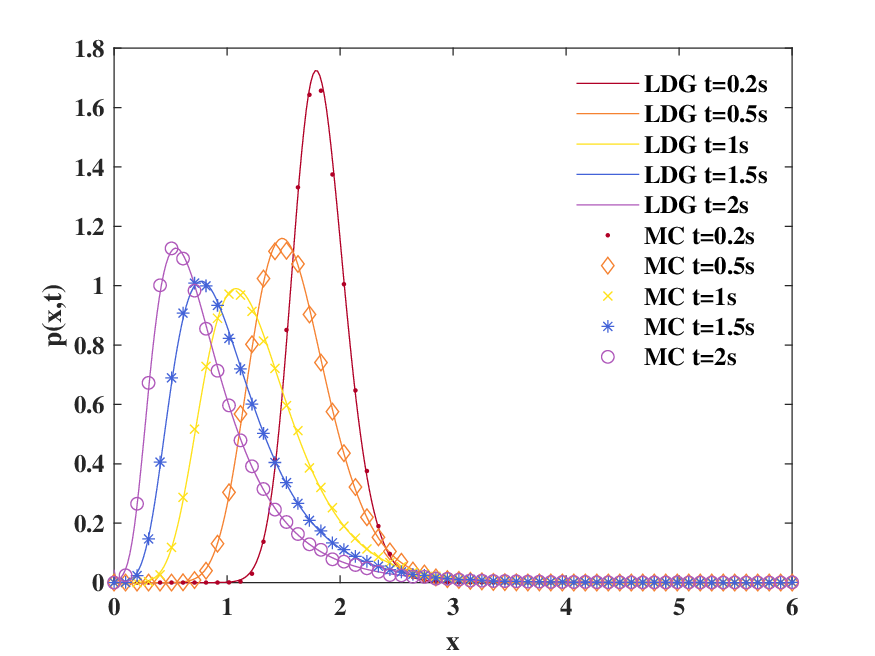}}
		\setcounter {subfigure} {0} (b){\includegraphics[scale = 0.5]{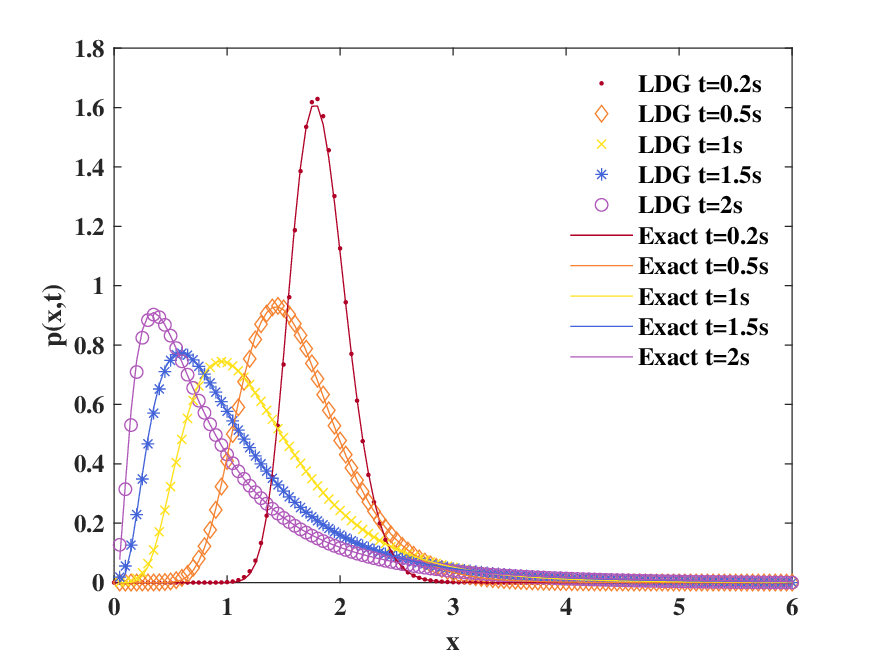}}
		\caption{PDF evolution of the memory-dependent PDEE (\ref{ldg1139}): (a) the comparisons among MC and LDG solutions with $ a=-0.5, b=0.25, c=0.25 $; (b) the comparisons among exact and LDG solutions with $ a=-0.5, b=0, c=0.5 $.}
		\label{fig3-9}
	\end{figure}

	We also calculate the $ L^2 $ and $ L^{\infty} $ error at the different times $t$ in $b=0$ case with parameters $ \Delta x=0.1, R=[0,6], \Delta t=0.0001 $ for both LDG and FD methods. The results are presented in Table \ref{tab3}. We find that the LDG method provides more accurate transient solutions compared to the FD method.

	\begin{table}[h]
		\caption{$ L^2 $ and $ L^{\infty} $ error between the LDG/FD solutions and the exact solutions of (\ref{ldg1139}) with $ a=-0.5, b=0, c=0.25, \Delta x=0.1,\Delta t=0.0001$.}\label{tab3}%, \Delta x=0.1,\Delta t=0.0001
		\begin{tabular*}{\textwidth}{@{\extracolsep\fill}lcccccc}
			\toprule%
			& \multicolumn{5}{@{}c@{}}{$ L^2 $ error} \\\cmidrule{2-6}%
			Method & $t=0.2$ s & $t=0.5$ s& $t=1.0$ s& $t=1.5$ s& $t=2.0$ s\\
			\midrule
			LDG  & 1.4133E-02 & 4.0188E-03 & 3.1496E-03 &  6.28630E-03 & 9.7419E-03   \\
			FD  & 5.0023E-02 & 1.3976E-02 & 9.1802E-03 &  1.2869E-02 & 2.2434E-02    \\
			\midrule
			& \multicolumn{5}{@{}c@{}}{$ L^{\infty} $ error} \\\cmidrule{2-6}%
			Method & $t=0.2$ s & $t=0.5$ s& $t=1.0$ s& $t=1.5$ s& $t=2.0$ s\\
			\midrule
			LDG & 2.3293E-02 & 6.3233E-03 & 4.6164E-03 &  1.2641E-02 & 1.5727E-02   \\
			FD  & 7.9791E-02 & 2.1090E-02 & 1.4716E-02 &  2.4770E-02 & 5.3389E-02  \\
			\bottomrule
		\end{tabular*}
	\end{table}
	
	To proceed, we consider the linear Langevin equation with time-dependent coefficients, i.e., $ f(t,x)=atx, g(t,x)=b\sqrt{t}x, h(t,x)=ct^dx$. Eq. \eqref{mSDEline} is understood as following
	\begin{equation}\label{example1-6}
	\mathrm{d}X_t= atX_t\mathrm{d}t+b\sqrt{t}X_t\circ\mathrm{d}W_t+ct^dX_t\circ \mathrm{d}B^H_t,\qquad X_0=x_0,
	\end{equation}
	where $ a, b, c $ and $ d $ are constants. By Theorem \ref{fplinax}, the PDFs of (\ref{example1-6}) satisfy the following memory-dependent PDEE	\begin{equation}\label{ldg1136}
	\frac{\partial }{\partial t}p(x,t)=-\frac{\partial }{\partial x}\big\{\big(atx+\frac{1}{2}b^2tx+\hat{C}_tx\big)p(x,t)\big\}+\frac{\partial^2 }{\partial x^2}\big\{\big(\frac{1}{2}b^2tx^2+\hat{C}_tx^2\big)p(x,t)\big\},
	\end{equation}
	where $ \hat{C}_t=c^2t^{2d+2H-1}H\frac{\Gamma(2H)\Gamma(1+d)}{\Gamma(d+2H)} $ and the initial condition is $ p (x, 0) = \delta (x-x_0) $.
	
    Using the LDG method, (\ref{ldg1136}) is numerically calculated with the parameters $ \Delta x=0.1, \Delta t=0.0004 $ and $ a=-0.25, b=0.25, c=0.25, d=0.8, x_0=2$. The LDG solutions at different times $t$ are calculated and compared with MC solutions in Fig. \ref{fig3-6}, which indicates the good accuracy of the numerical scheme.
	
	\begin{figure}[t!]
		\centering
		\subfigure{\includegraphics[scale = 1]{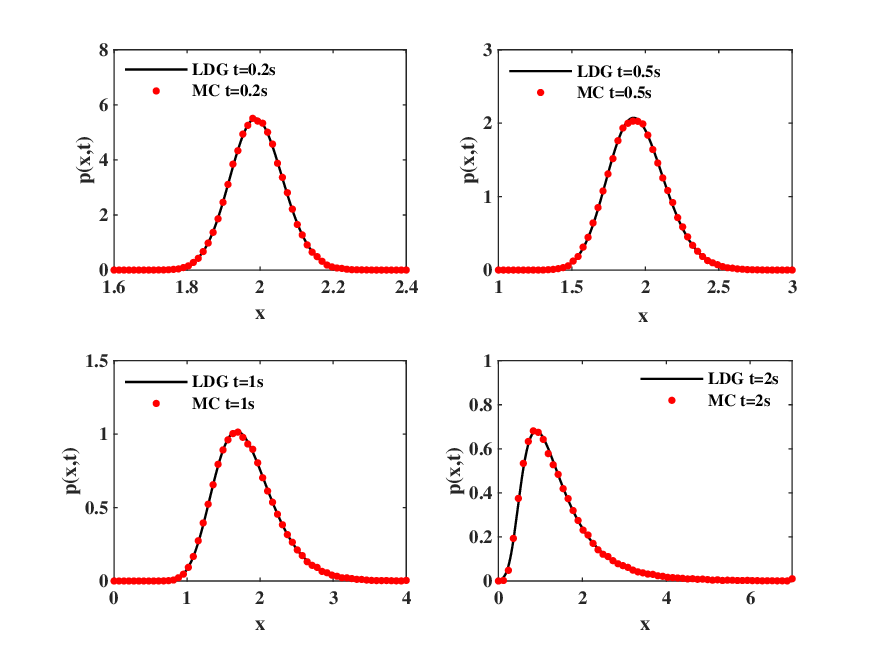}}
		\caption{PDF evolution of the memory-dependent PDEE (\ref{ldg1136}) with $ a=-0.25, b=0.25, c=0.25, d=0.8$.}
		\label{fig3-6}
	\end{figure}
	
	\subsection{The nonlinear Langevin equation excited by combined FGN and GWN}
	We first consider the nonlinear Langevin equation excited by combined multiplicative FGN and GWN. Specifically, the functions are defined as $ f(t,x)=a(x-dx^3), g(t,x)=b(x-dx^3) $, and $ h(t,x)=c(x-dx^3)$. Consequently, the nonlinear SDEs (\ref{mSDEscom}) can be expressed as follows
	\begin{equation}\label{example1-10}
	\mathrm{d}X_t= a(X_t-dX_t^3)\mathrm{d}t+b(X_t-dX_t^3)\circ\mathrm{d}W_t+c(X_t-dX_t^3)\circ\mathrm{d}B^H_t,\qquad X_0=x_0,
	\end{equation}
	where $ a, b, c $ and $ d $ are constants. By Theorem \ref{fpcomm}, the PDFs of (\ref{example1-10}) satisfy the following memory-dependent PDEE
	\begin{align}\label{ldg11310}
	\frac{\partial }{\partial t}p(x,t)=&-\frac{\partial }{\partial x}\big\{\big(a(x-dx^3)+\big(\frac{1}{2}b^2+Ht^{2H-1}c^2\big)(1-3dx^2)(x-dx^3)\big)p(x,t)\big\}\cr
    &+\frac{\partial^2 }{\partial x^2}\big\{\big(\frac{1}{2}b^2+Ht^{2H-1}c^2\big)(x-dx^3)^2p(x,t)\big\},
	\end{align}
	with the initial condition $ p (x, 0) = \delta (x-x_0) $.
	
	The parameters $ \Delta x=0.025, R=[0,1.5], \Delta t=0.0011 $ and $ a=-1, b=c=d=0.5, H=0.8, x_0=0.4$ are selected to calculate the LDG solutions at different times $t$. These solutions are compared with MC solutions in Fig. \ref{fig3-10} (a). The PDF evolution surface of the LDG solutions is shown in Fig. \ref{fig3-10} (b), demonstrating that the LDG scheme performs well even in the nonlinear case.
	
	\begin{figure}[H]
		\centering
		\setcounter {subfigure} {0} (a){\includegraphics[scale = 0.5]{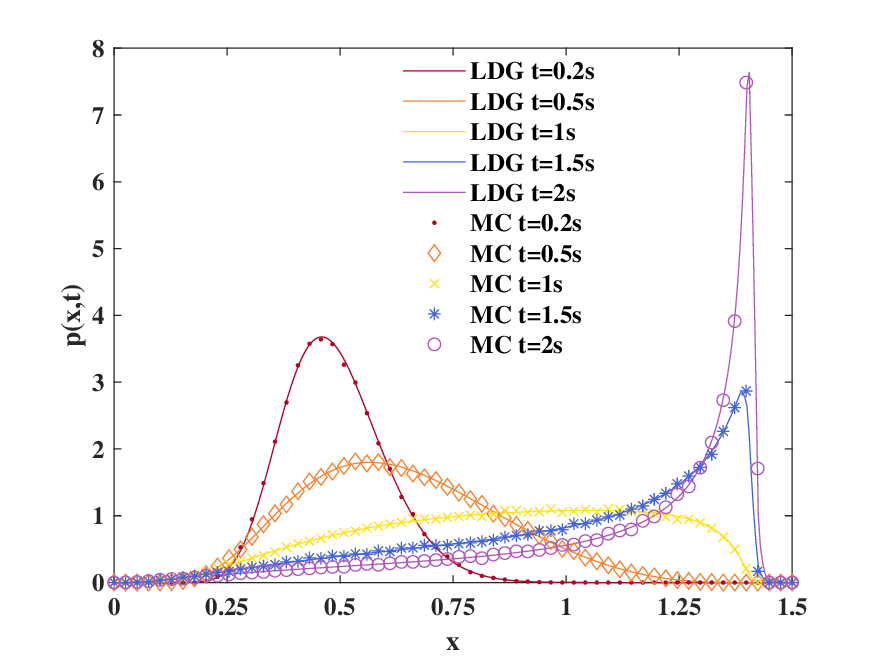}}
		\setcounter {subfigure} {0} (b){\includegraphics[scale = 0.5]{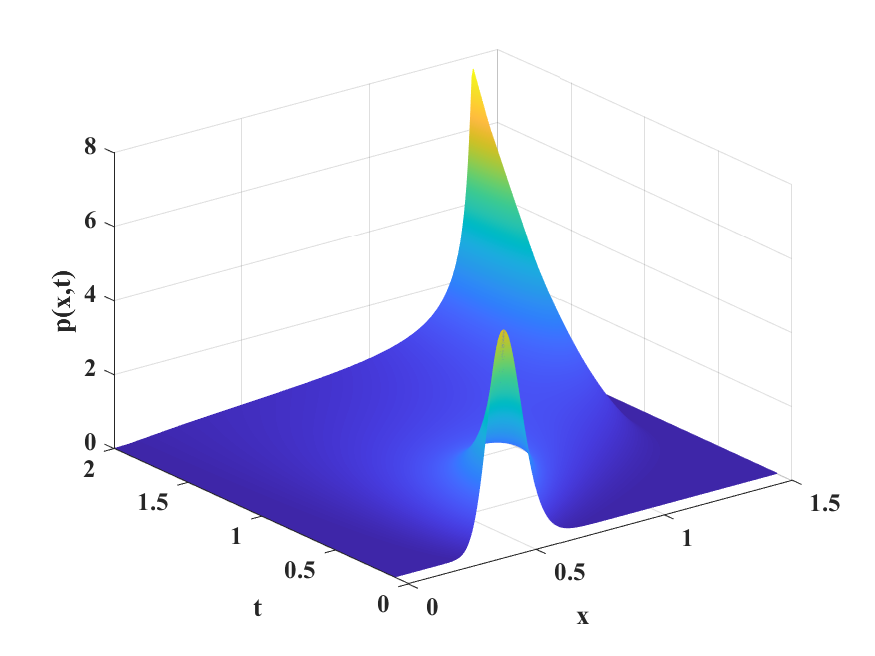}}
		\caption{The comparisons between MC and LDG solutions (a) of the memory-dependent PDEE (\ref{ldg11310}) and PDF evolution surface of the LDG solutions (b) with $a=-1, b=c=d=0.5, H=0.8, x_0=0.4$.}
		\label{fig3-10}
	\end{figure}
	
	To close this section, we consider the nonlinear Langevin equation excited by multiplicative FGN without drift term, i.e., $f(t,x)=g(t,x)\equiv0, h(t,x)=\sqrt{1+\sigma x^2}$, the nonlinear SDEs (\ref{SDEscomxh}) is understood as following
	\begin{equation}\label{example1-8}
	\mathrm{d}X_t= \sqrt{1+\sigma X_t^2}\circ\mathrm{d}B^H_t,\qquad X_0=x_0.
	\end{equation}
	By Corollary \ref{fpcomm-onlyh}, we obtain the following memory-dependent PDEE for PDFs of (\ref{example1-8})
	\begin{equation}\label{ldg1138}
	\frac{\partial }{\partial t}p(x,t)=-\frac{\partial }{\partial x}\big\{Ht^{2H-1}\sigma x p(x,t)\big\}+
	\frac{\partial^2 }{\partial x^2}\big\{Ht^{2H-1}(1+\sigma x^2)p(x,t)\big\},
	\end{equation}
	with the initial condition $ p (x, 0) = \delta (x-x_0) $. The exact solution is
	\begin{equation*}
	p(x,t)=C_1\frac{1}{\sqrt{\frac{1}{2}t^{2H}(1+\sigma x^2)}}\exp\bigg\{-\frac{\big(\ln(\sqrt{\sigma}x+\sqrt{1+\sigma x^2}\big)^2}{2t^{2H}\sigma}\bigg\},
	\end{equation*}
	where $C_1$ is the normalization constant.

	We select the parameters $ \Delta x=0.1, \Delta t=0.0006 $ and $\sigma=0.1, H=0.8, x_0=0$ to numerically calculate the memory-dependent PDEE (\ref{ldg1138}) by using LDG method. The LDG solutions at different times $t$ are compared with the MC and exact solutions in Fig. \ref{fig3-8}, demonstrating good agreement.
	
	\begin{figure}[t!]
		\centering
		\setcounter {subfigure} {0} (a){\includegraphics[scale = 0.50]{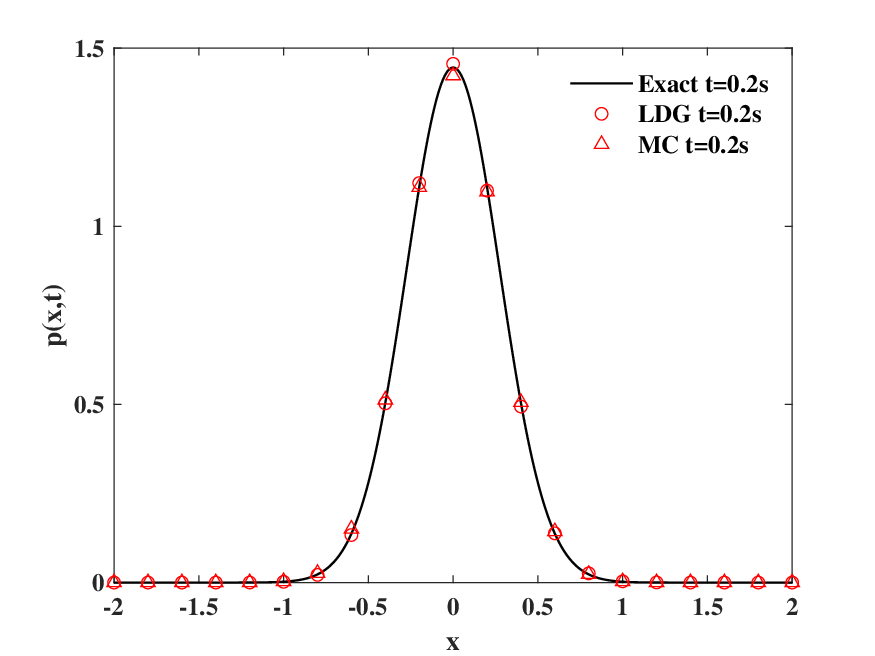}}
		\setcounter {subfigure} {0} (b){\includegraphics[scale = 0.50]{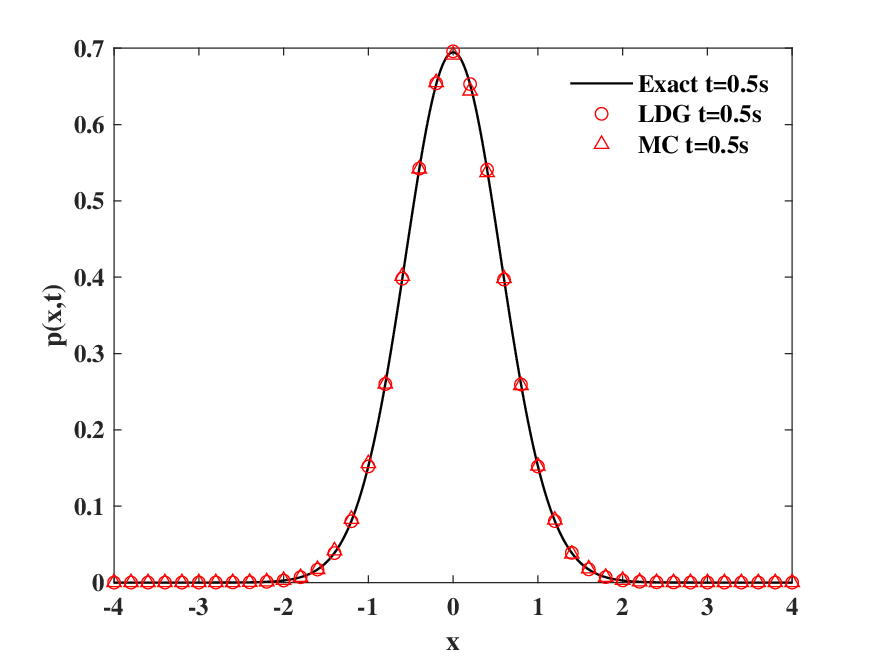}}
		\setcounter {subfigure} {0} (c){\includegraphics[scale = 0.50]{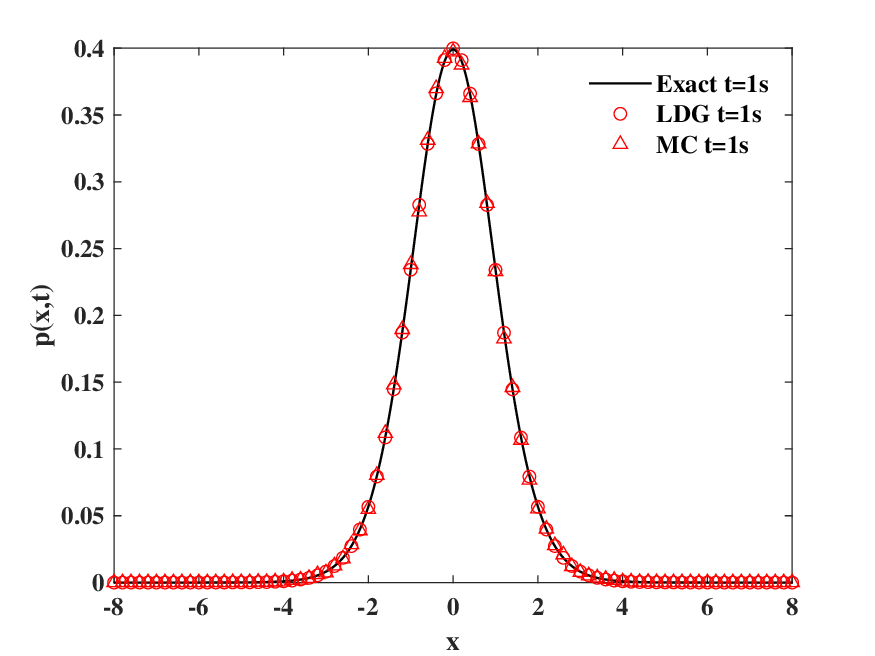}}
		\setcounter {subfigure} {0} (d){\includegraphics[scale = 0.50]{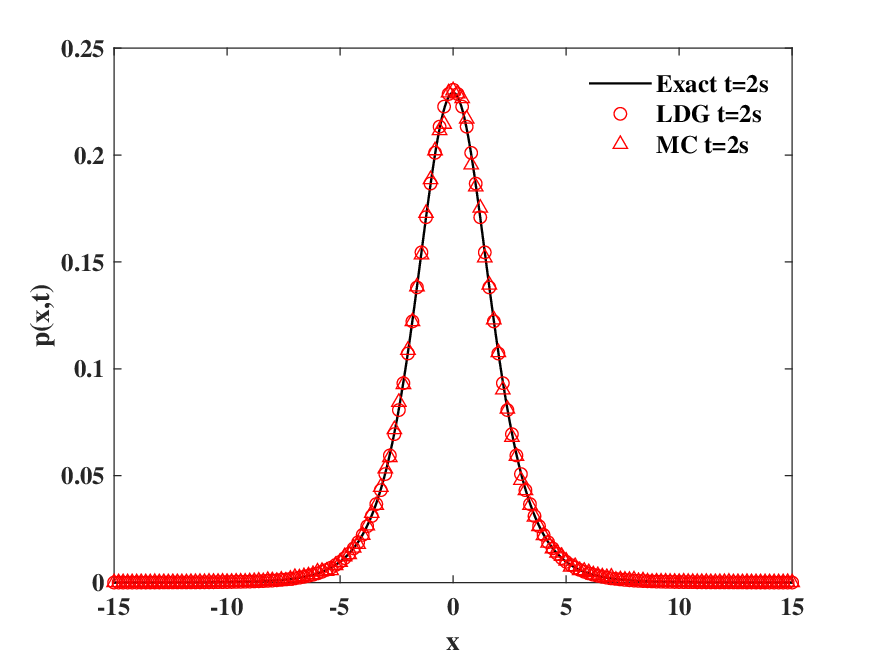}}
		\caption{The comparisons among the exact, MC and LDG solutions of the memory-dependent PDEE (\ref{ldg1138}) for different times $t$ with $\sigma=0.1, H=0.8, x_0=0$.}
		\label{fig3-8}
	\end{figure}
	
	For a fixed time $t$, we investigate the PDF evolution caused by FGN and the LDG solutions of the memory-dependent PDEEs (\ref{ldg11310}) and (\ref{ldg1138}) for different Hurst parameters at $ t=0.5$ s. The results are shown in Fig. \ref{fig3-10-2} (a) and (b), demonstrating the efficiency of the LDG method for nonlinear Langevin equations with varying memory properties.
	
	\begin{figure}[t!]
		\centering
		\setcounter {subfigure} {0} (a){\includegraphics[scale = 0.5]{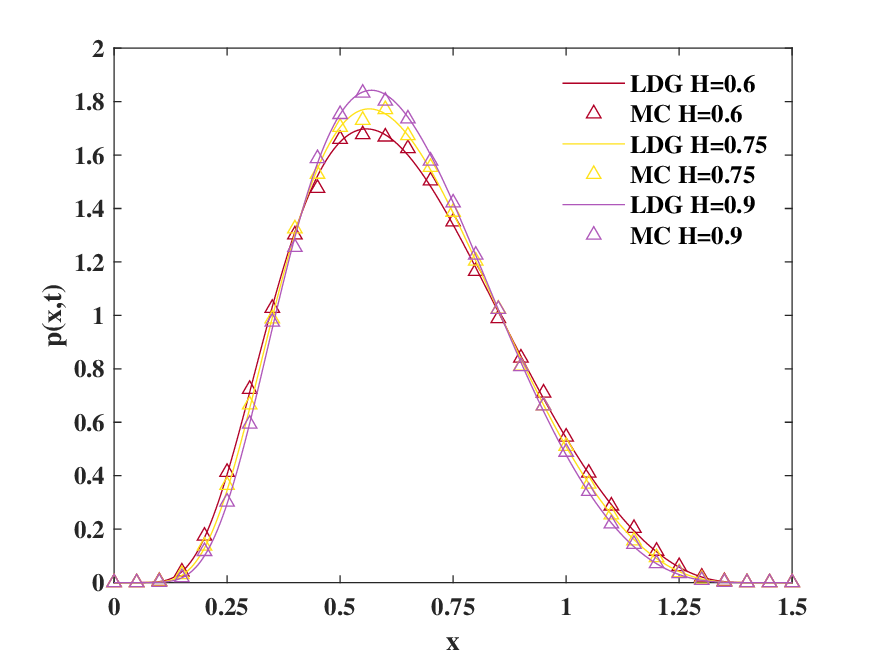}}
		\setcounter {subfigure} {0} (b){\includegraphics[scale = 0.5]{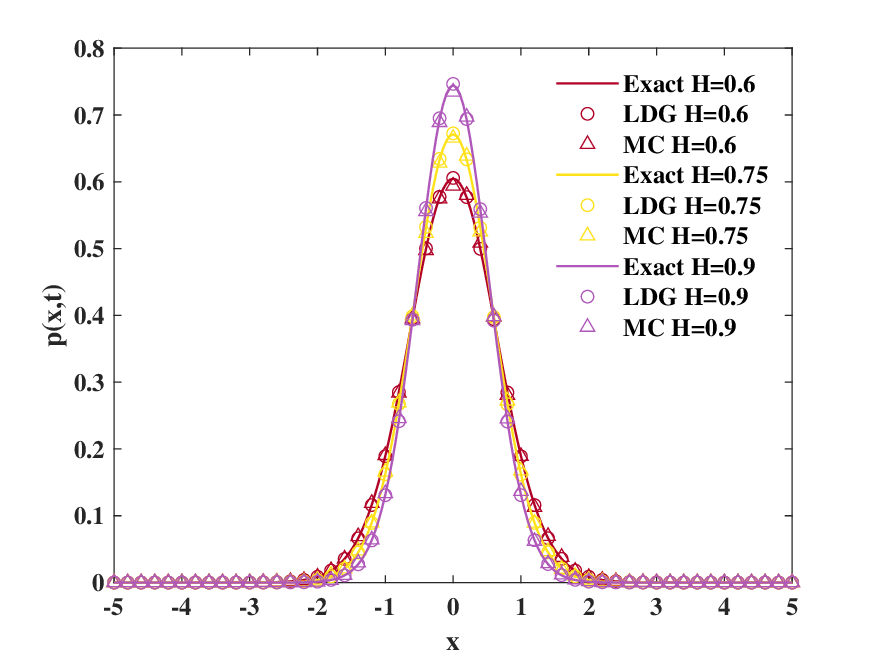}}
		\caption{The comparisons among MC and LDG solutions (a) of the memory-dependent PDEE (\ref{ldg11310}) and the comparisons among exact, MC and LDG solutions (b) of the memory-dependent PDEE (\ref{ldg1138}) for different Hurst parameters at $ t=0.5 $ s.}
		\label{fig3-10-2}
	\end{figure}
	
	\section{Discussions and conclusions}\label{sec13}
	By using the FWIS integral and rough path theory, a novel non-Markovian PDEM to derive the memory-dependent PDEEs is presented for the PDFs of the nonlinear Langevin equation excited by combined FGN and GWN. 
   It is a breakthrough to non-Markovian stochastic dynamics. However, the current non-Markovian PDEM still has some limitations for one-dimension and the commutative conditions setting. In future, we will try to obtain the memory-dependent PDEEs for PDFs of the nonlinear SDEs with high dimensions and much more general drift and diffusion terms without the commutativity conditions.
    Additionally, the LDG method is developed to address the numerical solutions of the memory-dependent PDEEs. To compare with the PI and FD methods, we demonstrate that the LDG solutions exhibit higher precision, as quantified by $ L^2 $ and $ L^{\infty} $ errors relative to the exact solutions used as benchmarks. Notably, this study represents the pioneering application of the LDG method to solve the memory-dependent PDEEs for the PDFs of the nonlinear Langevin equation excited by combined FGN and GWN. It is worth emphasizing that the LDG method not only obtains stationary solutions but also provides transient solutions with elevated precision. In contrast, the PI method struggles to solve the memory-dependent PDEEs for the non-Markovian nonlinear Langevin equation excited by FGN and exhibits more limitations when high-precision results are required.

	\backmatter
	
	\bmhead{Acknowledgements}
	B. Pei and L. Feng  were supported by National Natural Science Foundation (NSF) of China under Grant No. 12172285, Guangdong Basic and Applied Basic Research Foundation under Grant No. 2022A1515011853, and Shaanxi Fundamental Science Research Project for Mathematics and Physics under Grant No. 22JSQ027. Y. Li was supported by NSF of Shandong Province under Grant No. ZR2023ZD35, NSF of China under Grant No. 12301566, Science and Technology Commission of Shanghai Municipality under Grant No. 23JC1400300, Chenguang Program of Shanghai Education Development Foundation and Shanghai Municipal Education Commission under Grant No. 22CGA01. Y. Xu was supported by the Key International (Regional) Joint Research Program of the NSF of China under Grant No. 12120101002.
	
	\bmhead{Data availability}
	Data sharing is not applicable to this article as no datasets were generated or analyzed during the current study.
	
	\bmhead{Declarations}
	No.
	\bmhead{Conflict of interest}
	The authors declare that they have no conflict of interest.
	
	\begin{appendices}
		
		{\color{blue}	\section{The Proof of Lemmas}\label{secA1}}
		In this section, we will give lemma proofs that are not given in the paper.
		\subsection{The Proof of Lemma \ref{Malianx}}\label{app-Malianx}
		
		\begin{proof}
			By applying Lemma \ref{sym-path-to-fwick}, we convert (\ref{mSDEline}) into integral form
			\begin{align*} \mathrm{d}X_t&=A_tX_t\mathrm{d}t+B_tX_t\circ\mathrm{d}W_t+C_tX_t\circ \mathrm{d}B^H_t\cr
			&=A_tX_t\mathrm{d}t+B_tX_t\circ\mathrm{d}W_t+C_tX_t\diamond\mathrm{d}B^H_t+C_tD^{\phi}_tX_t\mathrm{d}t.
			\end{align*}
			Using Lemma \ref{mal-to-fwick} for fixed $r$, we get for all $t\in[0,T]$:
			\begin{align}\label{Malz}
			\mathrm{d}D^{\phi}_rX_t=& A_tD^{\phi}_rX_t\mathrm{d}t+B_tD^{\phi}_rX_t\circ\mathrm{d}W_t+C_tD^{\phi}_rX_t\diamond\mathrm{d}B^H_t\cr
			&+\phi(r,t)C_tX_t\mathrm{d}t+C_tD^{\phi}_r(D^{\phi}_tX_t)\mathrm{d}t,
			\end{align}
			
			To solve $D^{\phi}_tX_t$, define $z_t := \rho(r,t)X_t$, where $\rho(r,t)$ is deterministic and differentiable with $\rho(r,0) = 0$. Applying Lemma \ref{mal-to-fwick} again, we calculate $z_t$ as follows
			\begin{align}\label{Malztila}
			\mathrm{d}z_t=&\rho_t(r,t)X_t\mathrm{d}t+\rho(r,t)A_tX_t\mathrm{d}t+\rho(r,t)B_tX_t\circ\mathrm{d}W_t+\rho(r,t)C_tX_t\diamond\mathrm{d}B^H_t+\rho(r,t)C_tD^{\phi}_tX_t\mathrm{d}t\cr
			=&\rho_t(r,t)X_t\mathrm{d}t+A_tz_t\mathrm{d}t+B_tz_t\circ\mathrm{d}W_t+C_tz_t\diamond\mathrm{d}B^H_t+\rho(r,t)C_tD^{\phi}_tX_t\mathrm{d}t.
			\end{align}
			
			Comparing (\ref{Malztila}) with (\ref{Malz}), we conclude
			\begin{align*}
			\rho(r,t)=\int_{0}^{t}\phi(r,s)C_s\mathrm{d}s.
			\end{align*}
			and
			\begin{align*}
C_tD^{\phi}_r(D^{\phi}_tX_t)&=C_tD^{\phi}_r\Big(X_t\int_{0}^{t}\phi(t,s)C_s\mathrm{d}s\Big)=C_t\int_{0}^{t}\phi(t,s)C_s\mathrm{d}sD^{\phi}_rX_t=C_t\int_{0}^{t}\phi(t,s)C_s\mathrm{d}s\rho(r,t)X_t\cr
			&=\rho(r,t)C_tD^{\phi}_tX_t.
			\end{align*}
			Then $ D^{\phi}_tX_t= z_t$ due to the uniqueness of (\ref{Malz}) \cite{mishura2012mixed}. Thus, we have
			\begin{equation*}
			D^{\phi}_tX_t= X_t\int_{0}^{t}\phi(t,s)C_s\mathrm{d}s, \, \forall t\in[0,T], \, a.s.
			\end{equation*}
		\end{proof}

		\subsection{The Proof of Lemma \ref{itolinax}}\label{app-itolinax}
		\begin{proof}
			We convert the Stratonovich integral $\int X \circ \mathrm{d}W$ to the It\^o integral $\int X \mathrm{d}W$. Consequently, equation (\ref{mSDEline}) becomes equivalent to
			\begin{equation}\label{SDEslinex-ito}
			\mathrm{d}X_t=\big(A_t X_t+\frac{1}{2}B^2_t X_t\big) \mathrm{d}t+B_t X_t \mathrm{d}W_t+C_t X_t \circ \mathrm{d}B^H_t.
			\end{equation}
			
			By Theorem 3.1 in \cite{sonmez2023mixed}, and using Lemmas \ref{sym-path-to-fwick} and \ref{Malianx}, we obtain
			\begin{align*} 
			\mathrm{d}F(t,X_t)=&\Big(\frac{\partial F}{\partial t}(t,X_t)+\big(A_tX_t+\frac{1}{2}B^2_t X_t\big)\frac{\partial F}{\partial x}(t,X_t) +\frac{1}{2}B^2_tX_t^2\frac{\partial^2 F}{\partial x^2}(t,X_t)\Big)\mathrm{d}t\cr
			&+B_tX_t\frac{\partial F}{\partial x}(t,X_t)\mathrm{d}W_t
			+C_tX_t\frac{\partial F}{\partial x}(t,X_t) \circ \mathrm{d}B^H_t\cr
			=&\Big(\frac{\partial F}{\partial t}(t,X_t)+\big(A_tX_t+\frac{1}{2}B^2_t X_t\big)\frac{\partial F}{\partial x}(t,X_t) +\frac{1}{2}B^2_tX_t^2\frac{\partial^2 F}{\partial x^2}(t,X_t)\Big)\mathrm{d}t\cr
			&+B_tX_t\frac{\partial F}{\partial x}(t,X_t)\mathrm{d}W_t
			+C_tX_t\frac{\partial F}{\partial x}(t,X_t) \diamond \mathrm{d}B^H_t\cr
			&+\Big(C_tD^{\phi}_tX_t\frac{\partial F}{\partial x}(t,X_t)+C_tX_t D^{\phi}_tX_t\frac{\partial^2 F}{\partial x^2}(t,X_t)\Big)\mathrm{d}t\cr
			=&\Big(\frac{\partial F}{\partial t}(t,X_t)+\big(A_tX_t+\frac{1}{2}B^2_tX_t+\hat{C}_tX_t\big)\frac{\partial F}{\partial x}(t,X_t) +\big(\frac{1}{2}B^2_tX_t^2+\hat{C}_tX_t^2\big)\frac{\partial^2 F}{\partial x^2}(t,X_t)\Big)\mathrm{d}t\cr
			&+B_tX_t\frac{\partial F}{\partial x}(t,X_t)\mathrm{d}W_t
		+C_tX_t\frac{\partial F}{\partial x}(t,X_t) \diamond\mathrm{d}B^H_t,
			\end{align*}
			where $ \hat{C}_t= C_t\int_{0}^{t}\phi(t,r)C_r\mathrm{d}r$.	
		\end{proof}
		
		\subsection{The Proof of Lemma \ref{fplinax}}\label{app-fplinax}
		\begin{proof}
			Let $G(x)$ be a twice-differentiable function. By substituting $F(t, x) = G(x)$ into Lemma \ref{itolinax}, taking expectations on both sides, we obtain
			\begin{align}\label{fp1}
			\frac{\mathrm{d}}{\mathrm{d}t}\mathbb{E}[G(X_t)]=&\mathbb{E}\big[\big(A_tX_t+\frac{1}{2}B^2_tX_t+\hat{C}_tX_t\big)G'(X_t) \big]+\mathbb{E}\big[\big(\frac{1}{2}B^2_tX_t^2+\hat{C}_tX_t^2\big)G''(X_t)\big]
			=:\Sigma_1.
			\end{align}
			
			Using integration by parts, Fubini's theorem, and the fact that $p(-\infty,t) = p(\infty,t) = 0$, the right-hand side of equation (\ref{fp1}) can be expressed as
			\begin{align*}
			\Sigma_1=&
			\int_{\mathbb{R}}\Big\{\big(A_tx+\frac{1}{2} B^2_tx+\hat{C}_tx\big)p(x,t)\Big\}G'(x) \mathrm{d}x
			+\int_{\mathbb{R}}\Big\{\big(\frac{1}{2}B^2_tx^2+\hat{C}_tx^2\big)p(x,t)\Big\}G''(x)\mathrm{d}x\cr
			=&
			-\int_{\mathbb{R}}G(x)\frac{\partial }{\partial x}\Big\{\big(A_tx+\frac{1}{2} B^2_tx+\hat{C}_tx\big)p(x,t)\Big\}\mathrm{d}x-\int_{\mathbb{R}}G'(x)\frac{\partial }{\partial x} \Big\{\big(\frac{1}{2}B^2_tx^2+\hat{C}_tx^2\big)p(x,t)\Big\}\mathrm{d}x\cr
			=&
			-\int_{\mathbb{R}}G(x)\frac{\partial }{\partial x}\Big\{\big(A_tx+\frac{1}{2} B^2_tx+\hat{C}_tx\big)p(x,t)\Big\}\mathrm{d}x+\int_{\mathbb{R}}G(x)\frac{\partial^2 }{\partial x^2} \Big\{\big(\frac{1}{2}B^2_tx^2+\hat{C}_tx^2\big)p(x,t)\Big\}\mathrm{d}x.
			\end{align*}
			
			Since $ G(x) $ is arbitrary, so we get the desired result.
		\end{proof}
		
		\subsection{The Proof of Lemma \ref{sde-to-rde}}\label{app-sde-to-rde}
		\begin{proof}
			Rewriting the Stratonovich integral  as It\^o integral with Wong-Zakai correction,
			\begin{align}\label{mSDEscom-ito}
			\mathrm{d}X^{\mathrm{I}}_t= f^{\mathrm{I}}(X^{\mathrm{I}}_t)\mathrm{d}t+g(X^{\mathrm{I}}_t) \mathrm{d}W_t+h(X^{\mathrm{I}}_t) \circ  \mathrm{d}B^H_t,
			\end{align}
			where $f^{\mathrm{I}}(\cdot) = f(\cdot) + \frac{1}{2} g(\cdot) \frac{\partial g}{\partial x}(\cdot)$, and $X^{\mathrm{I}}$ denote the solutions of (\ref{mSDEscom-ito}) in the Itô sense. Then, the solutions $X$ of the SDEs (\ref{mSDEscom}) is equivalent to the solutions of the Itô SDEs (\ref{mSDEscom-ito}) in the mean square sense.
			
			Using a similar technique as in \cite[Theorem 4]{neuenkirch2015relation}, the solutions $X^{\mathrm{R}}$ of the RDEs (\ref{rde1}) and $X^{\mathrm{I}}$ of the Itô  SDEs (\ref{mSDEscom-ito}) coincide in probability under the given coefficients:
			$$ f^{\mathrm{R}}(\cdot)=f^{\mathrm{I}}(\cdot)-\frac12g(\cdot)\frac{\partial g}{\partial x}(\cdot), g^{\mathrm{R}}(\cdot)=g(\cdot), h^{\mathrm{R}}(\cdot)=h(\cdot).$$
			Thus, by passing to a subsequence, we obtain the desired result.
		\end{proof}
		
		\subsection{The Proof of Lemma \ref{flowuniq}}\label{app-flowuniq}
		\begin{proof}
			Since the vector fields $V_i$ commute, their flows $(e^{tV_i})_{t \in \mathbb{R}}$ also commute. Define 
			\begin{equation*}
			F(x_0,t,y)=\big(e^{tV_0}\circ e^{y_1V_1}\circ \cdots \circ e^{y_dV_d}\big)(x_0),
			\end{equation*}
			for $(x_0, t, y) \in \mathbb{R}^n \times \mathbb{R} \times \mathbb{R}^d$.
			
			Applying the change of variable formula, we see that the process $\big(e^{B_t^{H,d}V_d}(x_0)\big)_{t \geq 0}$ is a solution to
			\begin{equation*}
			\mathrm{d}\big(e^{B_t^{H,d}V_d}(x_0)\big)=V_d\big(e^{B_t^{H,d}V_d}(x_0)\big)\mathrm{d}B^{H,d}_t,
			\end{equation*}
			and the process $\big(e^{tV_0}(x_0)\big)_{t \geq 0}$ is a solution to
			\begin{equation*}
			\mathrm{d}\big(e^{tV_0}(x_0)\big)=V_0\big(e^{tV_0}(x_0)\big)\mathrm{d}t.
			\end{equation*}
			
			Using It\^o's formula and the commutativity of $V_{d-1}$ and $V_d$, we have
			\begin{align*}
			\mathrm{d}\big(e^{B_t^{H,d-1}V_{d-1}}(e^{B_t^{H,d}V_d}(x_0))\big)=&V_{d-1}\big(e^{B_t^{H,d-1}V_{d-1}}(e^{B_t^{H,d}V_d}(x_0))\big)\mathrm{d}B_t^{H,d-1}\cr
			&+V_d\big(e^{B_t^{H,d-1}V_{d-1}}(e^{B_t^{H,d}V_d}(x_0))\big)\mathrm{d}B_t^{H,d},
			\end{align*}
			
			Similarly, since $V_0$ and $V_1$ commute,
			\begin{align*}
			\mathrm{d}\big(e^{tV_0}(e^{B_t^{H,1}V_1}(x_0))\big)=V_0\big(e^{tV_0}(e^{B_t^{H,1}V_1}(x_0))\big)\mathrm{d}t+V_1\big(e^{tV_0}(e^{B_t^{H,1}V_1}(x_0))\big)\mathrm{d}B_t^{H,1}.
			\end{align*}
			
			By iteratively applying the change of variable formula, the process $(F(t, x_0, B_t^H))_{t \geq 0}$ satisfies
			\begin{equation*}
			\mathrm{d}F(t,x_0,B^H_t)=V_0(F(t,x_0,B^H_t))\mathrm{d}t+\sum_{i=1}^{d}V_i(F(t,x_0,B^H_t))\mathrm{d}B^{H,i}_t.
			\end{equation*}
			
			Thus, by pathwise uniqueness for Eq. (\ref{RDEvec}), we conclude that
			\begin{equation*}
			X_t=F(t,x_0,B^H_t),\quad X_0=x_0,t\geq0.
			\end{equation*}
		\end{proof}
		
		\subsection{The Proof of Lemma \ref{fpcomm}}\label{app-fpcomm}
		\begin{proof}
			Let $G(x)$ be any twice-differentiable function. By Lemma \ref{flowuniq} and the commutativity conditions, we have
			\begin{align}\label{phipde}
			\frac{\mathrm{d} }{\mathrm{d} t}\mathbb{E}[G(X_t)]=&\mathbb{E}\big[f(X_t)G'(X_t)\big]+\frac{1}{2}\mathbb{E}\big[g(X_t)\big(g'(X_t)G'(X_t)+g(X_t)G''(X_t)\big)\big]\cr
			&+Ht^{2H-1}\mathbb{E}\big[h(X_t)\big(h'(X_t)G'(X_t)+h(X_t)G''(X_t)\big)\big].
			\end{align}
			
			Following a similar approach to Lemma \ref{fplinax}, we obtain the desired result.
		\end{proof}
		
	\end{appendices}

\end{document}